\documentclass[onefignum,onetabnum]{siamart190516}

\usepackage{lipsum}
\usepackage{amsfonts}
\usepackage{amsmath}
\usepackage{bbm}
\usepackage{mathrsfs}
\usepackage{graphicx}
\usepackage{epstopdf}
\usepackage{algorithmic}
\usepackage{physics}
\usepackage{mathrsfs}
\usepackage{dsfont}
\usepackage{enumerate}
\usepackage{csquotes}
\usepackage{url}
\usepackage{color}
\usepackage{yhmath}

\ifpdf
  \DeclareGraphicsExtensions{.eps,.pdf,.png,.jpg}
\else
  \DeclareGraphicsExtensions{.eps}
\fi

\newsiamremark{hypothesis}{Hypothesis}
\crefname{hypothesis}{Hypothesis}{Hypotheses}
\newsiamthm{claim}{Claim}

\newtheorem{assumption}{Assumption}
\newtheorem{remark}{Remark}

\headers{LQ Stackelberg Games for Jump Diffusions}{Jun Moon}

\title{Linear-Quadratic Stochastic Stackelberg Differential Games for Jump-Diffusion Systems\thanks{Submitted to the editors DATE.
\funding{This research was supported in part by the National Research Foundation of Korea (NRF) Grant funded by the Ministry of Science and ICT, South Korea (NRF-2017R1E1A1A03070936, NRF-2017R1A5A1015311).
}}}

\author{Jun Moon\thanks{Department of Electrical Engineering, Hanyang University, Seoul 04763, South Korea (\email{junmoon@hanyang.ac.kr}, \url{https://junmoony.github.io/}).}
}

\usepackage{amsopn}

\ifpdf
\hypersetup{
  pdftitle={Linear-Quadratic Stochastic Stackelberg Differential Games for Jump-Diffusion Systems},
  pdfauthor={Jun Moon}
}
\fi

\begin{document}

\maketitle

\begin{abstract}
This paper considers linear-quadratic (LQ) stochastic leader-follower Stackelberg differential games for jump-diffusion systems with random coefficients. We first solve the LQ problem of the follower using the stochastic maximum principle and obtain the state-feedback representation of the open-loop optimal solution in terms of the integro-stochastic Riccati differential equation (ISRDE), where the state-feedback type control is shown to be optimal via the completion of squares method. Next, we establish the stochastic maximum principle for the indefinite LQ stochastic optimal control problem of the leader using the variational method. However, to obtain the state-feedback representation of the open-loop solution for the leader, there is a technical challenge due to the jump process. To overcome this limitation, we consider two different cases, in which the state-feedback type optimal control for the leader in terms of the ISRDE can be characterized by generalizing the Four-Step Scheme. Finally, in these two cases, we show that the state-feedback representation of the open-loop optimal solutions for the leader and the follower constitutes the Stackelberg equilibrium. Note that the (indefinite) LQ control problem of the leader is new and nontrivial due to the coupled FBSDE constraint induced by the rational behavior of the follower.
\end{abstract}

\begin{keywords}
  leader-follower Stackelberg game, LQ control for jump diffusions, forward-backward stochastic differential equation with jump diffusions, stochastic Riccati differential equation.
\end{keywords}

\begin{AMS}
 91A65, 93E20, 49K45, 49N10
\end{AMS}

\section{Introduction}
We first state the notation used in this paper. The precise problem formulation and the detailed literature review are then followed.
\subsection{Notation}
Let $\mathbb{R}^n$ be the $n$-dimensional Euclidean space. For $x,y \in \mathbb{R}^n$, $x^\top$ denotes the transpose of $x$, $\langle x,y \rangle$ is the inner product, and $|x| := \langle x, x \rangle^{1/2}$. Let $\mathbb{S}^n$ be the set of $n \times n$ symmetric matrices. Let $|x|^2_{S} :=x^\top S x$ for $x \in \mathbb{R}^n$ and $S \in \mathbb{S}^n$.


Let $(\Omega, \mathcal{F}, \mathbb{P})$ be a complete probability space with the natural filtration $\mathbb{F} :=\{\mathcal{F}_s,~ 0 \leq s \leq t\}$ generated by the following two mutually independent stochastic processes and augmented by all the $\mathbb{P}$-null sets in $\mathcal{F}$:
\begin{itemize}
\item a one dimensional standard Brownian motion $B$ defined on $[0,T]$;
\item an $E$-marked right continuous Poisson random measure (process) $N$ defined on $E  \times [0,T]$, where $E := \bar{E} \setminus \{0\}$ with $\bar{E} \subset \mathbb{R}$ is a Borel subset of $\mathbb{R}$ equipped with its Borel $\sigma$-field $\mathcal{B}(E)$. The intensity measure of $N$ is denoted by $\lambda^\prime(\dd e, \dd t) := \lambda(\dd e) \dd t$, satisfying $\lambda(E) < \infty$, where $\{\tilde{N}(A,(0,t]) := (N-\lambda^\prime)(A,(0,t])\}_{t \in (0,T]}$ is an associated compensated $\mathcal{F}_t$-martingale random (Poisson) measure of $N$ for any $A \in \mathcal{B}(E)$. Here, $\lambda$ is a $\sigma$-finite L\'evy measure on $(E,\mathcal{B}(E))$, which satisfies $\int_{E} (1 \wedge  |e|^2) \lambda (\dd e) < \infty$.\footnote{If the Poisson process $N$ has jumps of unite size ($E = \{1\}$), then $\{\tilde{N}((0,t]) := (N-\lambda^\prime)((0,t])\}_{t \in (0,T]}$ is the compensated Poisson process, where $\lambda^\prime (\dd t) := \lambda \dd t$ and $\lambda > 0$ is the intensity of $N$ \cite{Applebaum_book, Oksendal_book_jump, Privault_book}.}
\end{itemize}

We introduce the following spaces \cite{Applebaum_book}: for $t \in [0,T]$,
\begin{itemize}
\item $\mathcal{C}_{\mathbb{F}}^2(t,T;\mathbb{R}^n)$: the space of $\mathcal{F}_t$-adapted  $\mathbb{R}^n$-valued stochastic processes, which is c\`adl\`ag and satisfies $\|x\|_{\mathcal{C}_{\mathcal{F}}^2} := \mathbb{E}[\sup_{s \in [t,T]} |x(s)|^2]^{\frac{1}{2}} < \infty$;
\item $\mathcal{L}_{\mathbb{F}}^2(t,T;\mathbb{R}^n)$: the space of $\mathcal{F}_t$-adapted  $\mathbb{R}^n$-valued stochastic processes, satisfying $\|x\|_{\mathcal{L}_{\mathcal{F}}^2} := \mathbb{E}[\int_t^T |x(s)|^2 \dd s ]^{\frac{1}{2}} < \infty$;
\item $\mathcal{L}_{\mathbb{F},p}^2(t,T;\mathbb{R}^n)$: the space of $\mathcal{F}_t$-predictable $\mathbb{R}^n$-valued stochastic processes, satisfying $\|x\|_{\mathcal{L}_{\mathcal{F},p}^2} := \mathbb{E}[\int_t^T |x(s)|^2 \dd s ]^{\frac{1}{2}} < \infty$;
\item $G^2(E,\mathcal{B}(E),\lambda;\mathbb{R}^n)$:  the space of square integrable functions such that for $k \in G^2(E,\mathcal{B}(E),\lambda;\mathbb{R}^n)$, $k:E \rightarrow \mathbb{R}^n$ satisfies $\|k\|_{G^2} := (\int_{E} |k(e)|^2 \lambda(\dd e))^{\frac{1}{2}} < \infty$, where $\lambda$ is a $\sigma$-finite L\'evy measure on $(E,\mathcal{B}(E))$;
\item $\mathcal{G}^2_{\mathbb{F},p}(t,T,\lambda;\mathbb{R}^n)$: the space of stochastic processes such that $k:\Omega \times [t,T] \times E \rightarrow \mathbb{R}^n$, for $k \in \mathcal{G}^2_{\mathbb{F},p}(t,T,\lambda;\mathbb{R}^n)$, is an $\mathcal{P} \times \mathcal{B}(E)$-measurable $\mathbb{R}^n$-valued predictable process satisfying $\|k\|_{\mathcal{G}^2_{\mathbb{F},p}} := \mathbb{E}[\int_t^T \|k(s)\|_{G^2} \dd s]^{\frac{1}{2}} < \infty$, where $\mathcal{P}$ denotes the $\sigma$-algebra of $\mathcal{F}_t$-predictable subsets of $\Omega \times [t,T]$. 
\end{itemize}

\subsection{Problem Statement}\label{Section_1_2}
We consider the following controlled stochastic differential equation (SDE) on $[t,T]$\footnote{The assumption of the one-dimensional $B$ and $\widetilde{N}$ in (\ref{eq_1}) is only for notational convenience, and we can easily extend the results of this paper to the multi-dimensional case.}:
\begin{align}
\label{eq_1}
\begin{cases}
\dd x(s)  = \Bigl [ A(s) x(s-) + B_1(s) u_1(s) +  B_2(s) u_2(s) \Bigr ] \dd s \\
~~~~~~~~~ + \Bigl [ C(s) x(s-) + D_1(s) u_1(s) + D_2(s) u_2(s) \Bigr ] \dd B(s) \\
~~~~~~~~~ + \int_{E} \Bigl [ F(s,e) x(s-) + G_1(s,e) u_1(s) + G_2(s,e) u_2(s) \Bigr ] \widetilde{N}(\dd e, \dd s) \\
x(t) = a,
\end{cases}
\end{align}
where $x \in \mathbb{R}^n$ is the state process, $u_1 \in \mathbb{R}^{m_1}$ is the control of the leader, and $u_2 \in \mathbb{R}^{m_2}$ is the control of the follower. Let $\mathcal{U}_1 := \mathcal{L}_{\mathbb{F},p}^2(t,T;\mathbb{R}^{m_1})$ and $\mathcal{U}_2 := \mathcal{L}_{\mathbb{F},p}^2(t,T;\mathbb{R}^{m_2})$ be spaces of admissible controls for the leader and the follower, respectively. 
\begin{assumption}\label{Assumption_1}
$A,C:\Omega \times [0,T] \rightarrow \mathbb{R}^{n \times n}$, $B_i,D_i:\Omega \times [0,T] \rightarrow \mathbb{R}^{n \times m_i}$, $i=1,2$, $F:\Omega \times [0,T] \rightarrow G^2(E,\mathcal{B}(E),\lambda;\mathbb{R}^{n \times n})$, $G_i:\Omega \times [0,T] \rightarrow G^2(E,\mathcal{B}(E),\lambda;\mathbb{R}^{n \times m_i})$, $i=1,2$, are $\mathcal{F}_t$-predictable stochastic processes (random coefficients of (\ref{eq_1})), which are continuous in $t \in [0,T]$ and uniformly bounded in a.e $(\omega, t) \in \Omega \times [0,T]$.
\end{assumption}
Under Assumption \ref{Assumption_1}, for any $(u_1,u_2) \in \mathcal{U}_1 \times \mathcal{U}_2$, (\ref{eq_1}) admits a unique c\`adl\`ag solution in $\mathcal{C}_{\mathbb{F}}^2(t,T;\mathbb{R}^n)$ \cite[Theorem 6.2.3]{Applebaum_book} (see also \cite[Theorem 1.19]{Oksendal_book_jump} and \cite{Moon_MCRF_2020}).

\begin{remark}\label{Remark_1}
When the Poisson process $N$ has jumps of unit size, i.e., $E = \{1\}$, (\ref{eq_1}) becomes \cite{Applebaum_book, Oksendal_book_jump, Privault_book}
\begin{align*}
\begin{cases}
\dd x(s)  = \bigl [ A(s) x(s-) + B_1(s) u_1(s) +  B_2(s) u_2(s) \bigr ] \dd s \\
~~~~~~~~~ + \bigl [ C(s) x(s-) + D_1(s) u_1(s) + D_2(s) u_2(s) \bigr ] \dd B(s) \\
~~~~~~~~~ + \bigl [ F(s) x(s-) + G_1(s) u_1(s) + G_2(s) u_2(s) \bigr ] \dd \widetilde{N}(s) \\
x(t) = a.
\end{cases}
\end{align*}
\end{remark}

The objective functional to be minimized by the leader is given by
\begin{align}	
\label{eq_2}
J_1(a; u_1,u_2) & = \mathbb{E} \Bigl [ \int_t^T \bigl [ |x(s)|^2_{Q_1(s)}  + |u_1(s)|^2_{R_1(s)} \bigr ] \dd s + |x(T)|^2_{M_1} \Bigr ],
\end{align}
and the objective functional of the follower is as follows
\begin{align}
\label{eq_3}		
J_2(a; u_1,u_2) & = \mathbb{E} \Bigl [ \int_t^T \bigl [ |x(s)|^2_{Q_2(s)} + |u_2(s)|^2_{R_2(s)} \bigr ] \dd s + |x(T)|^2_{M_2} \Bigr ].
\end{align}

\begin{assumption}\label{Assumption_2}
$Q_i:\Omega \times [0,T] \rightarrow \mathbb{S}^{n}$ and $R_i:\Omega \times [0,T] \rightarrow \mathbb{S}^{m_i}$ $i=1,2$, are $\mathcal{F}_t$-predictable stochastic processes (random coefficients of (\ref{eq_2}) and (\ref{eq_3})), which are continuous in $t \in [0,T]$ and uniformly bounded in a.e $(\omega, t) \in \Omega \times [0,T]$. Also, $M_i:\Omega \rightarrow \mathbb{S}^{n}$, $i=1,2$, are $\mathcal{F}_T$-measurable random matrices, which are uniformly bounded in a.e. $\omega \in \Omega$. 
\end{assumption}

\begin{remark}\label{Remark_1_1}
Note that in Assumption \ref{Assumption_2}, the cost parameters $Q_i$, $R_i$, and $M_i$, $i=1,2$ of (\ref{eq_2}) and (\ref{eq_3}) are not needed to be positive (semi)definite matrices.
\end{remark}

The interaction between the leader and the follower of the LQ Stackelberg game of this paper can be stated as follows. The leader chooses and announces her (or his) optimal solution to the follower by considering the rational reaction of the follower. The follower then determines his (or her) optimal solution by responding to the optimal solution of the leader. Then the above problem can be referred to as the \emph{linear-quadratic (LQ) stochastic Stackelberg differential game for jump-diffusion systems with random coefficients}.

Under this setting, the problem can be solved in a reverse way \cite{Jiongmin_Yong_1, Bensoussan_SIAM, Basar2}. Specifically, the main objective of the follower is to minimize (\ref{eq_3}) subject to (\ref{eq_1}) for any control of the leader $u_1 \in \mathcal{U}_1$, i.e.,
\begin{align}
\label{eq_4}
\textbf{(LQ-F)}~~J_2(a;u_1,\overline{u}_2[a,u_1]) = \inf_{u_2 \in \mathcal{U}_2} J_2(a;u_1,u_2),~ \forall u_1 \in \mathcal{U}_1.
\end{align}
We note that from (\ref{eq_4}), $\overline{u}_2$ is an optimal strategy dependent on $(a,u_1) \in \mathbb{R}^n \times \mathcal{U}_1$, i.e., $\overline{u}_2 : \mathbb{R}^n  \times \mathcal{U}_1 \rightarrow \mathcal{U}_2$. Then given the optimal solution of \textbf{(LQ-F)}, the problem of the leader can be stated as follows:
\begin{align}	
\label{eq_5}
\textbf{(LQ-L)}~~J_1(a;\overline{u}_1,\overline{u}_2[a,\overline{u}_1]) = \inf_{u_1 \in \mathcal{U}_1} J_1(a;u_1,\overline{u}_2[a,u_1]).
\end{align}
When the pair $(\overline{u}_1,\overline{u}_2[a,\overline{u}_1]) \in \mathcal{U}_1 \times \mathcal{U}_2$ in (\ref{eq_4}) and (\ref{eq_5}) exists,  we say that the pair $(\overline{u}_1,\overline{u}_2[a,\overline{u}_1])$ constitutes an (adapted) open-loop type \emph{Stackelberg equilibrium} for the leader and the follower in the Stackelberg game \cite{Basar2, Jiongmin_Yong_1, Bensoussan_SIAM, Moon_Yang_TAC_2020}.


The main results of this paper can be summarized as follows:

In Section \ref{Section_2}, we solve \textbf{(LQ-F)} in (\ref{eq_4}). In particular, using the stochastic maximum principle for jump-diffusion systems \cite{Tang_SICON_1994}, we obtain an open-loop type optimal solution for \textbf{(LQ-F)} in terms of the forward-backward SDE (FBSDE) with jump diffusions and random coefficients, which explicitly depends on $(a,u_1) \in \mathbb{R}^n \times \mathcal{U}_1$. {\color{black}Since the open-loop type optimal solution is not implementable in practical situations, we obtain its state-feedback representation in terms of the integro-stochastic Riccati differential equation (ISRDE) by extending the Four-Step Scheme of \cite{Jiongmin_Yong_1} to the case of general jump-diffusion models.} We then show that the corresponding state-feedback type control is the optimal solution for \textbf{(LQ-F)} and identify the explicit optimal cost via the generalized completion of squares method (see Theorem \ref{Theorem_1}).

We solve \textbf{(LQ-L)} in (\ref{eq_5}) in Section \ref{Section_3}. \textbf{(LQ-L)} is the (indefinite) LQ stochastic optimal control problem for FBSDEs with jump diffusions and random coefficients, where the FBSDE constraint, induced from \textbf{(LQ-F)}, characterizes the rational reaction behavior of the follower \cite{Basar2, Bensoussan_SIAM} (see \textbf{(LQ-L)} in (\ref{eq_3_1})). We first obtain the stochastic maximum principle for \textbf{(LQ-L)} using the variational approach and duality analysis, which has not been studied in the existing literature. Then by the stochastic maximum principle, the open-loop optimal solution for \textbf{(LQ-L)} is obtained in terms of the coupled FBSDEs with jump diffusions and random coefficients (see Lemma \ref{Lemma_1}). 

The state-feedback representation of the open-loop optimal solution of \textbf{(LQ-L)} in terms of the ISRDE is obtained by establishing the Four-Step Scheme for FBSDEs with jump diffusions and random coefficients (see Theorems \ref{Theorem_2} and \ref{Theorem_3}). Unfortunately, there is a technical limitation, which did not appear in \cite{Jiongmin_Yong_1}. The detailed discussion is given in Remark \ref{Remark_7}. Hence, we consider two different cases: 
\begin{enumerate}[(i)]
	\item the Poisson process $N$ has jumps of unit size (see Remark \ref{Remark_1});
	\item the jump part of (\ref{eq_1}) does not depend on the control of the follower ($G_2 = 0$).
\end{enumerate}
We note that the Four-Step Schemes of (i) and (ii) are more involved than the Four-Step Scheme without jumps in \cite{Jiongmin_Yong_1}. A related discussion can be found in Section \ref{Section_1_3}. Moreover, the ISRDEs of \textbf{(LQ-L)} in (\ref{eq_3_21}) and (\ref{eq_3_31}) are nonsymmetric and highly nonlinear, whereas the SRDE in \cite[Theorem 3.3, (3.38)]{Jiongmin_Yong_1} is symmetric (see Remark \ref{Remark_7}). When \textbf{(LQ-F)} and \textbf{(LQ-L)} of (i) and (ii) are solvable, the corresponding open-loop optimal solutions constitute the Stackelberg equilibrium, and they admit the state-feedback representation (see Corollaries \ref{Corollary_1} and \ref{Corollary_2}).

\subsection{Literature Review and Main Contributions of this Paper}\label{Section_1_3}

A leader-follower Stackelberg differential game (in a deterministic setting) was first studied by H. Von Stackelberg in \cite{Stackelberg_book}. Since then, (deterministic and stochastic) Stackelberg differential games and their applications have been studied extensively in the literature, see \cite{Basar2, 1101986, 1101999, Freiling, Jiongmin_Yong_1, Shi_Automatica_2016, Bensoussan_SIAM, Srikant_large, LI_SICON_2018, Moon_Automatica_2018, Xu_SC_2018, Lin_TAC_2019, Zheng_DGAA_2019, Moon_Yang_TAC_2020} and the references therein. 

Specifically, a complete solution to the LQ stochastic Stackelberg differential game for SDEs without jump diffusions (equivalent to $F = G_1 = G_2 = 0$ in (\ref{eq_1})) was obtained by J. Yong in \cite{Jiongmin_Yong_1}. In \cite{Jiongmin_Yong_1}, the open-loop type Stackelberg equilibrium and its state-feedback representation in terms of the SRDE were obtained via the maximum principle and the standard Four-Step Scheme. A general stochastic maximum principle of Stackelberg differential games for SDEs without jumps was established in \cite{Bensoussan_SIAM} for both (adapted) open-loop and closed-loop information structures. Stochastic Stackelberg differential games for backward SDEs (BSDEs) (with deterministic coefficients) were studied in \cite{Zheng_DGAA_2019}. The authors in \cite{LI_SICON_2018} considered Stackelberg games for FBSDEs without jumps, and \cite{Xu_SC_2018} studied the delay case of SDEs (without jumps) with deterministic coefficients. Mean-field type stochastic Stackelberg differential games for SDEs without jumps were considered in \cite{Lin_TAC_2019, Moon_Yang_TAC_2020}.

{\color{black}Note that the references mentioned above considered the case of SDEs in a Brownian setting without jumps. To the best of our knowledge, a class of (LQ or nonlinear) stochastic Stackelberg differential games for general jump-diffusion systems with random coefficients has not been studied in the existing literature, and our paper addresses the LQ problem (see Section \ref{Section_1_2}). This paper can be viewed as a nontrivial extension of \cite{Jiongmin_Yong_1} to the problem for general jump-diffusion systems. The main generalizations and technical challenges of this paper compared with \cite{Jiongmin_Yong_1} are as follows:
\begin{enumerate}[(a)]
	\item the inclusion of the jump-diffusion part with random coefficients in (\ref{eq_1});
	\item the explicit dependence of the controls of the leader and the follower on the jump diffusion part of the SDE in (\ref{eq_1});
	\item the indefiniteness of the cost parameters in (\ref{eq_2}) and (\ref{eq_3}) (see Remark \ref{Remark_1_1}).
\end{enumerate}

In \textbf{(LQ-F)} of Section \ref{Section_2}, due to (a)-(c), the Four-Step Scheme to find the explicit state-feedback type optimal solution of \textbf{(LQ-F)} should be more involved than \cite[Section 2]{Jiongmin_Yong_1}. Specifically, in \textbf{(LQ-F)}, we need to obtain the equivalent state-feedback expression of the solution to the adjoint equation (BSDE with jumps diffusions and random coefficients (\ref{eq_2_1})) as in (\ref{eq_2_3}) and (\ref{eq_2_5}). However, due to the complex quadratic variations with respect to the Brownian motion and the (compensated) Poisson process, a more sophisticated analysis than \cite{Jiongmin_Yong_1} is needed to find (\ref{eq_2_3}) and (\ref{eq_2_5}), which leads to the explicit characterization of the ISRDE and the associated state-feedback type optimal control of \textbf{(LQ-F)} in Theorem \ref{Theorem_1}. We should mention that \cite[Section 2]{Jiongmin_Yong_1} is a special case of \textbf{(LQ-F)}. Note that in Theorem \ref{Theorem_1}, we also verify optimality of the state-feedback type optimal solution of \textbf{(LQ-F)} and obtain the explicit optimal cost, for which we additionally need to generalize the completion of squares method of \cite[Theorem 2.3]{Jiongmin_Yong_1} for general jump-diffusion models.

In \textbf{(LQ-L)} of Section \ref{Section_3}, the leader is faced with the indefinite LQ stochastic optimal control problem for FBSDEs with jump diffusions and random coefficients, where the constraint, the rational behavior of the follower, is induced from \textbf{(LQ-F)}. Due to (a)-(c), the stochastic maximum principle for \textbf{(LQ-L)} in Lemma \ref{Lemma_1} should be different from \cite[Theorem 3.2]{Jiongmin_Yong_1}. In particular, our adjoint processes in (\ref{eq_3_2}) are the coupled FBSDEs with jump diffusions and random coefficients. Moreover, in the proof of Lemma \ref{Lemma_1}, the duality analysis in the variational approach includes the stochastic integrals of the continuous (Brownian) and pure jump (compensated Poisson) martingale processes, and their associated quadratic variations, by which the additional integrals with respect to the L\'evy measure are included in the first-order optimality condition in (\ref{eq_3_4}). Such an extended duality analysis is not presented in \cite{Jiongmin_Yong_1}, and for the case without jumps, Lemma \ref{Lemma_1} degenerates to \cite[Theorem 3.2]{Jiongmin_Yong_1}.

The Four-Step Scheme for FBSDEs with jump diffusions and random coefficients is established in \textbf{(LQ-L)} of Section \ref{Section_3} to obtain the explicit state-feedback representation of the open-loop optimal solution in \textbf{(LQ-L)}, which should also be more challenging than \cite[Section 3]{Jiongmin_Yong_1}. Specifically, after defining augmented (forward and backward) state processes in (\ref{eq_3_10_1_1_1_1_1_1}) from Lemma \ref{Lemma_1}, a key step is to find the equivalent state-feedback expression of the solution to the augmented BSDE $(\mathcal{Y},\mathcal{Z},\mathcal{K})$ in  (\ref{eq_3_10}). Unfortunately, due to (a)-(c), there is a technical restriction, which did not appear in \cite{Jiongmin_Yong_1} (see Remark \ref{Remark_7}). To overcome this limitation, we identify two different conditions, under which the Four-Step Scheme for \textbf{(LQ-L)} can be established; see Sections \ref{Section_3_1} (Assumption \ref{Assumption_3}) and \ref{Section_3_2} (Assumption \ref{Assumption_4}) (or (i)-(ii) in Section \ref{Section_1_2}).

Regarding the Four-Step Scheme in Section \ref{Section_3_1}, due to (a)-(c), the cross-coupling nature of the quadratic variations with respect to the continuous and pure jump martingale processes are induced. Hence, as shown in the ISRDE of (\ref{eq_3_21}), several invertibility conditions of the block matrices are essentially required to get the state-feedback form of the solution to the augmented BSDE $(\mathcal{Y},\mathcal{Z},\mathcal{K})$ in (\ref{eq_3_10}). This leads to the characterization of the ISRDE and the associated state-feedback type optimal solution of \textbf{(LQ-L)} in Theorem \ref{Theorem_2}. Note that in \cite[Section 3]{Jiongmin_Yong_1}, as there is only one quadratic variation of the Brownian motion, such complex invertibility conditions did not appear. As for the Four-Step Scheme in Section \ref{Section_3_2}, although there are no such complex invertibility conditions due to Assumption \ref{Assumption_4} (see the ISRDE in (\ref{eq_3_31})), Section \ref{Section_3_2} still generalizes \cite[Section 3]{Jiongmin_Yong_1} to jump-diffusion models as seen from Theorem \ref{Theorem_3}. Related discussions are also given in Remarks \ref{Remark_6_1_1_1_1_1_1} and \ref{Remark_6_1_1_1_1_1_2}. Moreover, as stated in Remark \ref{Remark_7}, due to (a)-(c), the ISRDEs of the leader in (\ref{eq_3_21}) and (\ref{eq_3_31}) are not symmetric and highly nonlinear, whereas the SRDE in \cite[(3.38)]{Jiongmin_Yong_1} is symmetric.

In summary, our paper solves the LQ stochastic Stackelberg differential game for general jump-diffusion systems with random coefficients. We obtain the open-loop type Stackelberg equilibrium and its state-feedback representation by establishing the stochastic maximum principle and Four-Step Schemes with jump diffusions, which are nontrivial generalizations of the problem without jumps in \cite{Jiongmin_Yong_1}. We mention that the Four-Step Schemes of \textbf{(LQ-F)} and \textbf{(LQ-L)} have not been reported in the existing literature. One additional contribution of our paper is to identify the stochastic maximum principle for indefinite LQ stochastic control of FBSDEs with jump diffusions and random coefficients, which has not been studied in the existing literature. 

The organization of the paper is as follows. We solve \textbf{(LQ-F)} and \textbf{(LQ-L)} in Sections \ref{Section_2} and \ref{Section_3}, respectively, where the open-loop type Stackelberg equilibrium and its explicit state-feedback representation are obtained. In Section \ref{Section_4}, we discuss some special cases and possible extensions of this paper.

}

\section{LQ Stochastic Optimal Control for the Follower}\label{Section_2}
Suppose that $(\overline{x},\overline{u}_2)$ is the optimal solution of \textbf{(LQ-F)}. We introduce the adjoint equation:
\begin{align}
\label{eq_2_1}
\begin{cases}
	\dd p(s)  =  - \Bigl [ A(s)^\top p(s-) + C(s)^\top q(s) + Q_2(s) \overline{x}(s-)   \\
~~~~~~ + \int_{E} F(s,e)^\top r(s,e) \lambda (\dd e) \Bigr ] \dd s + q(s) \dd B(s) + \int_{E} r(s,e) \widetilde{N}(\dd e, \dd s),~ s \in [t,T) \\
p(T)  =  M_2 x(T).
\end{cases}
\end{align}
Note that (\ref{eq_2_1}) is the (linear) backward SDE (BSDE) with jump diffusions and random coefficients. There is a unique solution of (\ref{eq_2_1}) with $(p,q,r) \in \mathcal{C}_{\mathbb{F}}^2(t,T;\mathbb{R}^n) \times \mathcal{L}_{\mathbb{F}}^2(t,T;\mathbb{R}^n) \times \mathcal{L}_{\mathbb{F},p}^2(t,T;\mathbb{R}^n)$ \cite[Lemma 2.4]{Tang_SICON_1994} (see also \cite[Theorem 2.1]{Barles_SSR_1997}).

Based on the stochastic maximum principle in \cite[Theorem 2.1]{Tang_SICON_1994}, $\overline{u}_2$ satisfies the following first-order optimality condition:
\begin{align}
\label{eq_2_2}
B_2(s)^\top p(s-) + D_2(s)^\top q(s) + \int_{E} G_2(s,e)^\top r(s,e) \lambda (\dd e) + R_2(s) \overline{u}_2(s) = 0	.
\end{align}
We now consider the following transformation in the \emph{Four-Step Scheme}:
\begin{align}
\label{eq_2_3}
p(s) = P(s) \overline{x}(s) + \phi(s),
\end{align}
where $P \in \mathbb{S}^n$ with $P(T) = M_2$ and $\phi \in \mathbb{R}^n$ with $\phi(T) = 0$. 
Assume that $P$ and $\phi$ are of the following form:
\begin{align}
\label{eq_2_4}
\begin{cases}
\dd P(s) = \Lambda_1(s) \dd s + L(s) \dd B(s) + \int_{E} Z(s,e) \widetilde{N}(\dd e, \dd s),~ s \in [t,T)	 \\
\dd \phi(s) = \Lambda_2(s) \dd s + \theta(s) \dd B(s) + \int_{E} \psi(s,e)  \widetilde{N}(\dd e, \dd s),~ s \in [t,T),
\end{cases}
\end{align}
where $L,Z \in \mathbb{S}^n$ and $\theta,\psi \in \mathbb{R}^n$. Explicit expressions of (\ref{eq_2_4}) are obtained below.

By applying It\^o's formula for general L\'evy-type stochastic integrals \cite[Theorem 4.4.13]{Applebaum_book} to (\ref{eq_2_3}) and using (\ref{eq_2_4}), we have
\begin{align}
\label{eq_2_5_1}
\dd p(s) & = \bigl [ 	\Lambda_1(s) \dd s + L(s) \dd B(s) + \int_{E} Z(s,e) \widetilde{N}(\dd e, \dd s)\bigr ] \overline{x}(s-)  \\
& ~~~ +  P(s-) \bigl [ A(s) \overline{x}(s-) + B_1(s) u_1(s) +  B_2(s) \overline{u}_2(s) \bigr ] \dd s \nonumber \\
&~~~ +  P(s-) \bigl [ C(s) \overline{x}(s-) + D_1(s) u_1(s) + D_2(s) \overline{u}_2(s) \bigr ] \dd B(s) \nonumber \\
&~~~ + \int_{E}  P(s-) \bigl [ F(s,e) \overline{x}(s-) + G_1(s,e) u_1(s) + G_2(s,e) \overline{u}_2(s) \bigr ] \widetilde{N}(\dd e, \dd s)  \nonumber \\
&~~~ + L(s) \bigl [ C(s) \overline{x}(s-) + D_1(s) u_1(s) + D_2(s) \overline{u}_2(s) \bigr ] \dd s \nonumber \\
&~~~ + \int_{E} Z(s,e) \bigl [ F(s,e) \overline{x}(s-) + G_1(s,e) u_1(s) + G_2(s,e) \overline{u}_2(s) \bigr ] \lambda (\dd e) \dd s \nonumber \\
&~~~ + \int_{E} Z(s,e) \bigl [ F(s,e) \overline{x}(s-) + G_1(s,e) u_1(s) + G_2(s,e) \overline{u}_2(s) \bigr ] \widetilde{N}(\dd e, \dd s) \nonumber \\
&~~~ + \Lambda_2(s) \dd s + \theta(s) \dd B(s) + \int_{E} \psi(s,e)  \widetilde{N}(\dd e, \dd s) \nonumber \\
&= - \bigl [ A(s)^\top p(s-) + C(s)^\top q(s) + Q_2(s) \overline{x}(s-)   \nonumber \\
&~~~ + \int_{E} F(s,e)^\top r(s,e) \lambda (\dd e) \bigr ] \dd s + q(s) \dd B(s) + \int_{E} r(s,e) \widetilde{N}(\dd e, \dd s). \nonumber
\end{align}
Note the coefficients of $B$ and $\widetilde{N}$ in (\ref{eq_2_5_1}). Then $q$ and $r$ can be written as
\begin{align}
\label{eq_2_5}
\begin{cases}
	q(s) = L(s) \overline{x}(s-) + P(s-) \bigl [ C(s) \overline{x}(s-) + D_1(s) u_1(s) + D_2(s) \overline{u}_2(s) \bigr ]  + \theta(s) \\
r(s,e) = Z(s,e) \overline{x}(s-) + P(s-) \bigl [ F(s,e) \overline{x}(s-) + G_1(s,e) u_1(s) + G_2(s,e) \overline{u}_2(s) \bigr ]  \\
~~~~~~~~~~ + Z(s,e)\bigl [ F(s,e) \overline{x}(s-) + G_1(s,e) u_1(s) + G_2(s,e) \overline{u}_2(s) \bigr ]  + \psi(s,e). 
\end{cases}
\end{align}
Substituting (\ref{eq_2_3}) and (\ref{eq_2_5}) into (\ref{eq_2_2}) yields
\begin{align}
\label{eq_2_7}
\overline{u}_2(s)   & = - \widehat{R}_2(s)^{-1} \widehat{S}_2(s)^\top \overline{x}(s-) - \widehat{R}_2(s)^{-1}  f(s),
\end{align}
provided that $\widehat{R}_2$ is invertible, where
\begin{align}
\label{eq_2_9_1}
\begin{cases}
\widehat{R}_2(s) := R_2(s) + D_2(s)^\top P(s-) D_2(s) \\
 ~~~  + \int_{E} G_2(s,e)^\top (P(s-) + Z(s,e)) G_2(s,e) \lambda (\dd e) \\
\widehat{S}_2(s) := \bigl ( B_2(s)^\top P(s-) + D_2(s)^\top L(s) + D_2(s)^\top P(s-) C(s)   \\
 ~~~   + \int_{E} \langle G_2(s,e), Z(s,e) + P(s-) F(s,e) + Z(s,e) F(s,e)\rangle \lambda(\dd e) \bigr )^\top  \\
\widehat{S}_1(s) := D_2(s)^\top P(s-) D_1(s)   \\
 ~~~ +  \int_{E} \langle G_2(s,e), P(s-) G_1(s,e) + Z(s,e) G_1(s,e) \rangle \lambda (\dd e) \\
 f(s) := B_2(s)^\top \phi(s-) + D_2(s)^\top \theta(s) \\
 ~~~ + \int_{E} G_2(s,e)^\top \psi(s,e) \lambda (\dd e) + \widehat{S}_1(s) u_1(s).
\end{cases}
\end{align}
Note that (\ref{eq_2_7}) is the optimal control with the state-feedback representation, which explicitly depends on $u_1 \in \mathcal{U}_1$. We can easily see that $\overline{u}_2 \in \mathcal{U}_2$ for a fixed $u_1 \in \mathcal{U}_1$.

By substituting (\ref{eq_2_7}) into (\ref{eq_1}), we have
\begin{align}
\label{eq_2_8}
\begin{cases}
\dd \overline{x}(s)  = \Big [ \widehat{A}(s) \overline{x}(s-) + \widehat{B}_2(s) 	 \phi(s-) + \widehat{H}_2(s) \theta(s) \\
~~~ + \int_{E} \widehat{K}_2(s,e) \psi(s,e) \lambda (\dd e) + \widehat{B}_1(s) u_1(s) \Big ] \dd s \\
~~~ + \Big [ \widehat{C}(s) \overline{x}(s-) + \widehat{H}_2(s)^\top \phi(s-) + \widetilde{H}_2(s)\theta(s) \\
~~~ + \int_{E} \widetilde{K}_2(s,e)\psi(s,e)\lambda(\dd e) + \widehat{D}_1(s) u_1(s) \Big ] \dd B(s) \\
~~~ +  \int_{E} \Big[ \widehat{F}(s,e)  \overline{x}(s-) + \widehat{K}_2(s,e)^\top \phi(s-) + \widetilde{K}_2(s,e)^\top \theta(s) \\
~~~ + \int_{E} \overline{K}_2(s,e,e^\prime) \psi(s,e^\prime) \lambda (\dd e^\prime) + \widehat{G}_1(s,e) u_1(s) \Bigr ] \widetilde{N} (\dd e, \dd s),~ s \in (t,T] \\
x(t) = a,
\end{cases}
\end{align}
where
\begin{align}
\label{eq_2_10_1_1_1}
\begin{cases}
\widehat{A}(s) := A(s) - B_2(s) \widehat{R}_2(s)^{-1} \widehat{S}_2(s-)^\top,~ \widehat{B}_2(s) := - B_2(s)\widehat{R}_2(s)^{-1} B_2(s)^\top \\
\widehat{H}_2(s) := - B_2(s)\widehat{R}_2(s)^{-1} D_2(s)^\top,~ \widehat{K}_2(s,e) := - B_2(s) \widehat{R}_2(s)^{-1}G_2(s,e)^\top
\\
\widehat{B}_1(s) := B_1(s) - B_2(s) \widehat{R}_2(s)^{-1} \widehat{S}_1(s) \\
\widehat{C}(s) := C(s) - D_2(s) \widehat{R}_2(s)^{-1} \widehat{S}_2(s)^\top \\
\widetilde{H}_2(s) := -D_2(s) \widehat{R}_2(s)^{-1} D_2(s)^\top,~ \widetilde{K}_2(s,e) := - D_2(s) \widehat{R}_2(s)^{-1} G_2(s,e)^\top \\
\widehat{D}_1(s) := D_1(s) - D_2(s) \widehat{R}_2(s)^{-1}  \widehat{S}_1(s) \\
\widehat{F}(s,e) := F(s,e) -  G_2(s,e) \widehat{R}_2(s)^{-1} \widehat{S}_2(s)^\top \\
\overline{K}_2(s,e,e^\prime) := - G_2(s,e) \widehat{R}_2(s)^{-1}  G_2(s,e^\prime)^\top  \\
\widehat{G}_1(s,e) := G_1(s,e) - G_2(s,e) \widehat{R}_2(s)^{-1} \widehat{S}_1(s).
\end{cases}
\end{align}
Note that (\ref{eq_2_8}) is the rational behavior of the follower under the (state-feedback type) optimal control in (\ref{eq_2_7}).
 
Substituting (\ref{eq_2_3}), (\ref{eq_2_7}) and (\ref{eq_2_5}) into (\ref{eq_2_5_1}), we can show that (\ref{eq_2_4}) has to satisfy the following symmetric integro-stochastic Riccati differential equation (ISRDE):
\begin{align}
\label{eq_2_11}
\begin{cases}
	\dd P(s)  = - \Bigl [ A(s)^\top P(s-) + P(s-) A(s) + Q_2(s) + L(s) C(s) 	\\
~~~~~~~ + C(s)^\top L(s) + C(s)^\top P(s-) C(s) + \int_{E} [ Z(s,e) F(s,e)  \\
~~~~~~~ + F(s,e)^\top Z(s,e)  + F(s,e)^\top P(s-) F(s,e)  \\
~~~~~~~ + F(s,e)^\top Z(s,e) F(s,e)] \lambda (\dd e)  - \widehat{S}_2(s) \widehat{R}_2(s)^{-1} \widehat{S}_2(s)^\top  \Bigr ] \dd s \\
~~~ + L(s) \dd B(s) + \int_{E} Z(s,e) \widetilde{N}(\dd e, \dd s),~ s \in [t,T) \\
P(T) = M_2,
\end{cases}	
\end{align}
and $\phi$ in (\ref{eq_2_4}) is the following BSDE with jumps and random coefficients:
\begin{align}	
\label{eq_2_12}
\begin{cases}
\dd \phi(s)  = - \Bigl [  \widehat{A}(s)^\top \phi(s-)  + \widehat{C}(s)^\top \theta(s) + \int_{E} \widehat{F}(s,e)^\top \psi(s,e) \lambda(\dd e)  \\
~~~~~~~ + \widehat{H}_1(s)^\top u_1(s) + \int_{E} \widehat{K}_1(s,e)^\top u_1(s) \lambda( \dd e)  \Bigr ]\dd s  \\
~~~ + \theta (s) \dd B(s)  + \int_{E} \psi(s,e) \widetilde{N}(\dd e, \dd s),~ s \in [t,T) \\
\phi(T) = 0.
\end{cases}
\end{align}
with $\widehat{H}_1$ and $\widehat{K}_1$ defined by
\begin{align}
\label{eq_2_13_1_1_1}
\begin{cases}
\widehat{H}_1(s) :=	\bigl ( C(s)^\top P(s-) D_1(s) + P(s-) B_1(s) \\
~~~~~~~~~~~~~ + L(s) D_1(s) -  \widehat{S}_2(s-) \widehat{R}_{2}(s)^{-1} \widehat{S}_1(s) \bigr )^\top
\\
\widehat{K}_1(s,e) := \bigl ( F(s,e)^\top P(s-) G_1(s,e) \\
~~~~~~~~~~~~~ + F(s,e)^\top Z(s,e) G_1(s,e) + Z(s,e) G_1(s,e)  \bigr )^\top.
\end{cases}
\end{align}

In summary, we have the following result:
\begin{theorem}\label{Theorem_1}
Assume that Assumptions \ref{Assumption_1} and \ref{Assumption_2} hold. Suppose that $(P,L,Z) \in \mathcal{C}_{\mathbb{F}}^2(t,T;\mathbb{S}^n) \times \mathcal{L}_{\mathbb{F}}^2(t,T;\mathbb{S}^n) \times \mathcal{L}_{\mathbb{F},p}^2(t,T;\mathbb{S}^n)$ is the solution of the ISRDE in (\ref{eq_2_11}), and $(\phi,\theta,\psi) \in \mathcal{C}_{\mathbb{F}}^2(t,T;\mathbb{R}^n) \times \mathcal{L}_{\mathbb{F}}^2(t,T;\mathbb{R}^n) \times \mathcal{L}_{\mathbb{F},p}^2(t,T;\mathbb{R}^n)$ is the solution of the BSDE with jump diffusions in (\ref{eq_2_12}). Assume that $\widehat{R}_2$ defined in (\ref{eq_2_9_1}) is uniformly positive definite for a.e. $(\omega,s) \in \Omega \times [0,T]$. Then the state-feedback representation of the optimal control for \textbf{(LQ-F)} can be written as 
\begin{align}	
\label{eq_2_13}
\overline{u}_2(s)   & = - \widehat{R}_2(s)^{-1} \widehat{S}_2(s)^\top x(s-) - \widehat{R}_2(s)^{-1} f(s),
\end{align}
where $f$ is defined in (\ref{eq_2_9_1}). 
Moreover, the optimal cost of \textbf{(LQ-F)} under (\ref{eq_2_13}) is
\begin{align}
\label{eq_2_14}
& J_2(a; u_1,\overline{u}_2)  = \inf_{u_2 \in \mathcal{U}_2} J_2(a;u_1,u_2) \\
	& = \mathbb{E} \Bigl [|a|^2_{P(0)}  + 2 \langle a, \phi(0) \rangle  +  \int_t^T  |u_1(s)|^2_{D_1(s)^\top P(s-) D_1(s)} \dd s  \nonumber \\
	&~~~  + \int_t^T \int_{E} |u_1(s)|^2_{G_1(s,e)^\top P(s-) G_1(s,e) + G_1(s,e)^\top Z(s,e) G_1(s,e)} \lambda (\dd e) \dd s \nonumber \\
	&~~~ + 2 \int_t^T  \langle u_1(s), B_1(s) ^\top \phi(s-) + D_1(s)^\top \theta(s) \rangle  \dd s \nonumber \\
	&~~~ + 2 \int_t^T \int_{E} \langle u_1(s), G_1(s,e)^\top \psi(s,e) \rangle   \lambda (\dd e) \dd s - \int_t^T |f(s)|^2_{\widehat{R}_2(s)^{-1}} \dd s \Bigr ]. \nonumber
\end{align}
\end{theorem}
\begin{proof}
For a given $u_1 \in \mathcal{U}_1$, let $\overline{x}$ be the state process controlled by $\overline{u}_2$ in (\ref{eq_2_13}), which is equivalent to (\ref{eq_2_8}). Then from Assumptions \ref{Assumption_1} and \ref{Assumption_2}, for any $u_1 \in \mathcal{U}_1$, (\ref{eq_2_8}) admits a unique c\`adl\`ag solution in $\mathcal{C}_{\mathbb{F}}^2(t,T;\mathbb{R}^n)$ \cite{Applebaum_book,Oksendal_book_jump,Moon_MCRF_2020}. Since $(\phi,\theta,\psi)$ in (\ref{eq_2_12}) is a linear BSDE, it admits a unique solution of $(\phi,\theta,\psi) \in \mathcal{C}_{\mathbb{F}}^2(t,T;\mathbb{R}^n) \times \mathcal{L}_{\mathbb{F}}^2(t,T;\mathbb{R}^n) \times \mathcal{L}_{\mathbb{F},p}^2(t,T;\mathbb{R}^n)$ \cite[Lemma 2.4]{Tang_SICON_1994}. For a fixed $(a,u_1) \in \mathbb{R}^n \times \mathcal{U}_1$, it holds that $\overline{u}_2 \in \mathcal{U}_2$.

For any $u_2 \in \mathcal{U}_2$, we apply It\^o's formula for general L\'evy-type stochastic integrals  \cite[Theorem 4.4.13]{Applebaum_book} to $\dd \langle x(s), P(s) x(s) \rangle + 2 \dd \langle x(s), \phi(s) \rangle$, where $x$ is the SDE in (\ref{eq_1}), $(P,L,Z)$ is solution of the ISRDE in (\ref{eq_2_11}), and $(\phi,\theta,\psi)$ is the solution of the BSDE with jump diffusions in (\ref{eq_2_12}). Then ($\Tr(\cdot)$ denotes the trace operator)
\begin{align}
\label{eq_A_1}
& \dd \langle x(s),	P(s) x(s) \rangle \\
& = \dd \Tr ( P(s) x(s)x(s)^\top ) \nonumber \\
& = 2 \Bigl [ A(s) x(s-) + B_1(s) u_1(s) +  B_2(s) u_2(s) \Bigr ]^\top P(s-) x(s-)  \dd s \nonumber \\ 
&~~~ + \Bigl [ C(s) x(s-) + D_1(s) u_1(s) + D_2(s) u_2(s) \Bigr ]^\top \nonumber \\
&~~~~~~~~ \times P(s-) \Bigl [ C(s) x(s-) + D_1(s) u_1(s) + D_2(s) u_2(s) \Bigr ]  \dd s \nonumber \\ 
&~~~ + \int_{E} \Bigl [ F(s,e) x(s-) + G_1(s,e) u_1(s) + G_2(s,e) u_2(s) \Bigr ]^\top \nonumber \\
&~~~~~~~~ \times (P(s-) + Z(s,e)) \Bigl [ F(s,e) x(s-) + G_1(s,e) u_1(s) + G_2(s,e) u_2(s) \Bigr ] \lambda (\dd e) \dd s \nonumber \\
&~~~ + 2 \Bigl [  C(s) x(s-) + D_1(s) u_1(s) + D_2(s) u_2(s) \Bigr ]^\top L(s) x(s-) \dd s \nonumber \\
&~~~ + 2 \int_{E} \Bigl [ F(s,e) x(s-) + G_1(s,e) u_1(s) + G_2(s,e) u_2(s) \Bigr ]^\top Z(s,e) x(s-) \lambda (\dd e) \dd s \nonumber \\
&~~~ + x(s-)^\top \dd P(s) x(s-) + [\cdots] \dd B(s) + \int_{E} [\cdots] \widetilde{N}(\dd e, \dd s), \nonumber
\end{align}
and
\begin{align}
\label{eq_A_2}
& 2 \dd \langle x(s), \phi(s) \rangle \\
& = 2 \Bigl [ A(s) x(s-) + B_1(s) u_1(s) +  B_2(s) u_2(s) \Bigr ]^\top \phi(s-) \dd s \nonumber \\
&~~~ + 2 \Bigl [ C(s) x(s-) + D_1(s) u_1(s) + D_2(s) u_2(s) \Bigr ]^\top \theta(s) \dd s \nonumber \\
&~~~ + 2 \int_{E} \Bigl [ F(s,e) x(s-) + G_1(s,e) u_1(s) + G_2(s,e) u_2(s) \Bigr ]^\top \psi(s,e) \lambda (\dd e) \dd s \nonumber \\
&~~~ + 2 x(s-)^\top \dd \phi(s)   + [\cdots] \dd B(s) + \int_{E} [\cdots] \widetilde{N}(\dd e, \dd s). \nonumber
\end{align}
Note that under Assumptions \ref{Assumption_1} and \ref{Assumption_2}, the stochastic integrals of the Brownian motion and the compensated Poisson process in (\ref{eq_A_1}) and (\ref{eq_A_2}) are $\mathcal{F}_s$-martingales \cite[page 287]{Applebaum_book}. Hence, their expectations are zero.

By integrating (\ref{eq_A_1}) and (\ref{eq_A_2}) from $0$ to $T$ and taking the expectation, we have (note the definition of $J_2$ in (\ref{eq_3}), the terminal conditions of the ISRDE in (\ref{eq_2_11}) and the BSDE in (\ref{eq_2_12}), and recall (\ref{eq_2_9_1}))
\begin{align}
\label{eq_A_3}
& J_2(a;u_2,u_2) - \mathbb{E}\Bigl [\langle a, P(0) a \rangle + 2 \langle a, \phi(0) \rangle \Bigr ]  \\
& = \mathbb{E} \Bigl [ \int_t^T \langle u_2(s), \widehat{R}_2(s) u_2(s) + \widehat{S}_2(s)^\top x(s-) + f(s) \rangle \dd s \nonumber \\
&~~~ + \int_t^T  2 \langle \widehat{S}_2(s)^\top x(s-), \widehat{R}_2(s)^{-1} \widehat{S}_2(s)^\top x(s-) + R_2(s)^{-1} f(s) \rangle \dd s \nonumber \\
&~~~ + \int_t^T  |u_1(s)|^2_{D_1(s)^\top P(s-) D_1(s)} \dd s \nonumber \\
&~~~  + \int_t^T \int_{E} |u_1(s)|^2_{G_1(s,e)^\top P(s-) G_1(s,e) + G_1(s,e)^\top Z(s,e) G_1(s,e)} \lambda (\dd e) \dd s \nonumber \\
	&~~~ + 2 \int_t^T  \langle u_1(s), B_1(s) ^\top \phi(s-) + D_1(s)^\top \theta(s) \rangle  \dd s \nonumber \\
	&~~~ + 2 \int_t^T \int_{E} \langle u_1(s), G_1(s,e)^\top \psi(s,e) \rangle \lambda (\dd e) \dd s \Bigr ]. \nonumber
\end{align}
Then by completing the integrands in (\ref{eq_A_3}) with respect to $u_2$, $J_2$ can equivalently be written as follows: (recall (\ref{eq_2_9_1}))
\begin{align}	
\label{eq_2_15}
& J_2(a;u_1,u_2)  = \mathbb{E} \Bigl [|a|^2_{P(0)}  + 2 \langle a, \phi(0) \rangle     \\ 
	&~~~ + \int_t^T \Bigl | u_2(s) + \widehat{R}_2(s)^{-1} \widehat{S}_2(s)^\top x(s-) + \widehat{R}_2(s)^{-1} f(s) \Bigr |^2_{\widehat{R}_2(s)} \dd s \nonumber \\
	&~~~ +  \int_t^T  |u_1(s)|^2_{D_1(s)^\top P(s-) D_1(s)} \dd s \nonumber \\
	&~~~  + \int_t^T \int_{E} |u_1(s)|^2_{G_1(s,e)^\top P(s-) G_1(s,e) + G_1(s,e)^\top Z(s,e) G_1(s,e)} \lambda (\dd e) \dd s \nonumber \\
	&~~~ + 2 \int_t^T  \langle u_1(s), B_1(s) ^\top \phi(s-) + D_1(s)^\top \theta(s) \rangle  \dd s \nonumber \\
	&~~~ + 2 \int_t^T \int_{E} \langle u_1(s), G_1(s,e)^\top \psi(s,e) \rangle \lambda (\dd e) \dd s - \int_t^T |f(s)|^2_{\widehat{R}_2(s)^{-1}} \dd s \Bigr ].  \nonumber
\end{align}

Since $\widehat{R}_2 > 0$ for a.e. $(\omega,s) \in \Omega \times [0,T]$, for a given $u_1 \in \mathcal{U}_1$, we have
\begin{align}
\label{eq_2_16}
J_2(a;u_1,u_2) \geq 	J_2(a;u_1,\overline{u}_2),~ \forall u_2 \in \mathcal{U}_2.
\end{align}
This shows that (\ref{eq_2_13}) is the optimal control of \textbf{(LQ-F)}, and (\ref{eq_2_8}) is the corresponding optimal state trajectory. From (\ref{eq_2_15}) and (\ref{eq_2_16}), we can easily see that (\ref{eq_2_14}) is the optimal cost of \textbf{(LQ-F)}. This completes the proof.	
\end{proof}

\begin{remark}\label{Remark_3_1_1_1_1}
{\color{black}
\begin{enumerate}[(i)]
\item Recently, the well-posedness (solvability) of (\ref{eq_2_11}) is shown in \cite{Zhang_SICON_2020}, which considers the (one-player) LQ control problem for jump-diffusion models with random coefficients. Specifically, in view of \cite[Theorems 4.1 and 5.2]{Zhang_SICON_2020}, when $R_2$ is uniformly positive definite, and $Q_2$ and $M_2$ are positive semidefinite for a.e. $(\omega,s) \in \Omega \times [0,T]$, there exists a unique solution of (\ref{eq_2_11}) with $(P,L,Z) \in \mathcal{C}_{\mathbb{F}}^2(t,T;\mathbb{S}^n) \times \mathcal{L}_{\mathbb{F}}^2(t,T;\mathbb{S}^n) \times \mathcal{L}_{\mathbb{F},p}^2(t,T;\mathbb{S}^n)$. We note that for the case without jumps, the ISRDE in (\ref{eq_2_11}) degenerates to the SRDE of \cite[(1.5)]{Jiongmin_Yong_1}.
\item The proof of Theorem \ref{Theorem_1} is known as the completion of squares method, where a key argument is to obtain the equivalent objective functional of $J_2$ in (\ref{eq_2_15}) that is quadratic in $u_2$. The proof of Theorem \ref{Theorem_1} is more involved than that for the case without jumps in \cite[Theorem 2.3]{Jiongmin_Yong_1}, as in our case there exist several different quadratic variations with respect to the Brownian potion and the (compensated) Poisson process (see the complete proof in []).
\end{enumerate}
}
\end{remark}

\section{LQ Stochastic Optimal Control for the Leader}\label{Section_3}
This section addresses \textbf{(LQ-L)} in (\ref{eq_5}). Note that the optimization constraint of \textbf{(LQ-L)} is (\ref{eq_2_8}) and (\ref{eq_2_12}), which characterize the rational behavior of the follower under (\ref{eq_2_13}). That is, \textbf{(LQ-L)} can be rewritten as follows:
\begin{align}	
\label{eq_3_1}
\textbf{(LQ-L)}~~J_1(a;\overline{u}_1,\overline{u}_2) = \inf_{u_1 \in \mathcal{U}_1} J_1(a;u_1,\overline{u}_2),~ \text{subject to (\ref{eq_2_8}) and (\ref{eq_2_12})}.
\end{align}
\begin{remark}
{\color{black}It is easy to see that \textbf{(LQ-L)} is a class of indefinite LQ stochastic optimal control problems for FBSDEs with jump diffusions and random coefficients.}
\end{remark}

We first state the stochastic maximum principle for \textbf{(LQ-L)}:
\begin{lemma}\label{Lemma_1}
Suppose that Assumptions \ref{Assumption_1} and \ref{Assumption_2} hold. Let $\overline{u}_1 \in \mathcal{U}_1$, where $\overline{x}$ is the corresponding state trajectory. Let $(\overline{x},\beta) \in \mathcal{C}_{\mathbb{F}}^2(t,T;\mathbb{R}^n \times \mathbb{R}^n)$, $(\phi,\theta,\psi) \in \mathcal{C}_{\mathbb{F}}^2(t,T;\mathbb{R}^n) \times \mathcal{L}_{\mathbb{F}}^2(t,T;\mathbb{R}^n) \times \mathcal{L}_{\mathbb{F},p}^2(t,T;\mathbb{R}^n)$ and $(\alpha,\eta,\gamma) \in \mathcal{C}_{\mathbb{F}}^2(t,T;\mathbb{R}^n) \times \mathcal{L}_{\mathbb{F}}^2(t,T;\mathbb{R}^n) \times \mathcal{L}_{\mathbb{F},p}^2(t,T;\mathbb{R}^n)$ be the solution of the following coupled FBSDEs:
\begin{align}
\label{eq_3_2}
\begin{cases}
\dd \overline{x}(s)  = \Big [ \widehat{A}(s) \overline{x}(s-) + \widehat{B}_2(s) 	 \phi(s-) + \widehat{H}_2(s) \theta(s) \\
~~~ + \int_{E} \widehat{K}_2(s,e) \psi(s,e) \lambda (\dd e) + \widehat{B}_1(s) \overline{u}_1(s) \Big ] \dd s \\
~~~ + \Big [ \widehat{C}(s) \overline{x}(s-) + \widehat{H}_2(s)^\top \phi(s-) + \widetilde{H}_2(s)\theta(s) \\
~~~ + \int_{E} \widetilde{K}_2(s,e)\psi(s,e)\lambda(\dd e) + \widehat{D}_1(s) \overline{u}_1(s) \Big ] \dd B(s) \\
~~~ +  \int_{E} \Big[ \widehat{F}(s,e)  \overline{x}(s-) + \widehat{K}_2(s,e)^\top \phi(s-) + \widetilde{K}_2(s,e)^\top \theta(s) \\
~~~ + \int_{E} \overline{K}_2(s,e,e^\prime) \psi(s,e^\prime) \lambda (\dd e^\prime) + \widehat{G}_1(s,e) \overline{u}_1(s) \Bigr ] \widetilde{N} (\dd e, \dd s),~ s \in (t,T] \\
\dd \alpha(s) = - \Bigl [ \widehat{A}(s)^\top \alpha(s-) + 	\widehat{C}(s)^\top \eta(s) + \int_{E} \widehat{F}(s,e)^\top \gamma(s,e) \lambda (\dd e) \\
~~~ + Q_1(s) \overline{x}(s-)  \Bigr ] \dd s + \eta(s) \dd B(s) + \int_{E} \gamma(s,e) \widetilde{N}(\dd e, \dd s),~ s \in [t,T)\\
\dd \phi(s)  = - \Bigl [  \widehat{A}(s)^\top \phi(s-)  + \widehat{C}(s)^\top \theta(s) + \int_{E} \widehat{F}(s,e)^\top \psi(s,e) \lambda(\dd e) + \widehat{H}_1(s)^\top \overline{u}_1(s) \\
~~~ + \int_{E} \widehat{K}_1(s,e)^\top \overline{u}_1(s) \lambda( \dd e)  \Bigr ]\dd s + \theta (s) \dd B(s)  + \int_{E} \psi(s,e) \widetilde{N}(\dd e, \dd s),~ s \in [t,T) \\
\dd \beta(s)  = \Bigl [ \widehat{A}(s)\beta(s-) + \widehat{B}_2(s) \alpha(s-) + \widehat{H}_2(s) \eta(s) + \int_{E} \widehat{K}_2(s,e) \gamma(s,e) \lambda(\dd e) \Bigr ] \dd s \\
~~~ + \Bigl [ \widehat{C}(s)\beta(s-) + \widehat{H}_2(s)^\top \alpha(s-) + \widetilde{H}_2(s) \eta(s) + \int_{E} \widetilde{K}_2(s,e) \gamma(s,e) \lambda(\dd e) \Bigr ] \dd B(s) \\
~~~ + \int_{E} \Bigl [ \widehat{F}(s,e)\beta(s-) + \widehat{K}_2(s,e)^\top \alpha(s-) \\
~~~~~~~~~~ + \widetilde{K}_2(s,e)^\top \eta(s) + \int_{E} \overline{K}_2(s,e,e^\prime) \gamma(s,e^\prime) \lambda (\dd e^\prime) \Bigr ]\widetilde{N}(\dd e, \dd s),~ s \in (t,T] \\
\overline{x}(t) = a,~ \beta(t) = 0,~ \phi(T) = 0,~ \alpha(T) = M_1 \overline{x}(T).
\end{cases}	
\end{align}
For $u_1^\prime \in \mathcal{U}_1$, let $(x^\prime,\beta^\prime) \in \mathcal{C}_{\mathbb{F}}^2(t,T;\mathbb{R}^n \times \mathbb{R}^n)$, $(\phi^\prime,\theta^\prime,\psi^\prime) \in \mathcal{C}_{\mathbb{F}}^2(t,T;\mathbb{R}^n) \times \mathcal{L}_{\mathbb{F}}^2(t,T;\mathbb{R}^n) \times \mathcal{L}_{\mathbb{F},p}^2(t,T;\mathbb{R}^n)$ and $(\alpha^\prime,\eta^\prime,\gamma^\prime) \in \mathcal{C}_{\mathbb{F}}^2(t,T;\mathbb{R}^n) \times \mathcal{L}_{\mathbb{F}}^2(t,T;\mathbb{R}^n) \times \mathcal{L}_{\mathbb{F},p}^2(t,T;\mathbb{R}^n)$ be the coupled FBSDEs in (\ref{eq_3_2}), where the initial condition holds $(x^\prime(t),\beta^\prime(t),\phi^\prime(t),\alpha^\prime(t)) =  (0,0,0,M_1 x^\prime(T))$. Assume that the following holds:
\begin{align}
\label{eq_3_3}
& \mathbb{E} \Bigl [ \int_0^T  \langle u_1^\prime(s), R_1(s) u_1^\prime(s) \rangle + \Bigl \langle u_1^\prime(s), \widehat{B}_1(s)^\top \alpha^\prime(s-) + \widehat{D}_1(s)^\top \eta^\prime(s) \\
&~~~ + \int_{E} \widehat{G}_1(s)^\top \gamma^\prime(s,e) \lambda (\dd e) + \widehat{H}_1(s) \beta^\prime(s-) + \int_{E} \widehat{K}_1(s,e) \beta^\prime(s-) \lambda (\dd e) \Bigr \rangle \dd s \Bigr ] \geq 0. \nonumber
\end{align}
Then $\overline{u}_1 \in \mathcal{U}_1$ is the optimal control for \textbf{(LQ-L)} if and only if the following first-order optimality condition holds:
\begin{align}
\label{eq_3_4}
& R_1(s) \overline{u}_1(s) + \widehat{B}_1(s)^\top \alpha(s-) + \widehat{D}_1(s)^\top \eta(s) \\
&~~~ + \int_{E} \widehat{G}_1(s)^\top \gamma(s,e) \lambda (\dd e) + \widehat{H}_1(s) \beta(s-) + \int_{E} \widehat{K}_1(s,e) \beta(s-) \lambda (\dd e)  = 0. \nonumber	
\end{align}
\end{lemma}
\begin{proof}
We note that $(\overline{x},\phi,\theta,\psi) \in \mathcal{C}_{\mathbb{F}}^2(t,T;\mathbb{R}^n) \times \mathcal{C}_{\mathbb{F}}^2(t,T;\mathbb{R}^n) \times \mathcal{L}_{\mathbb{F}}^2(t,T;\mathbb{R}^n) \times \mathcal{L}_{\mathbb{F},p}^2(t,T;\mathbb{R}^n)$ admits a unique solution in view of Theorem \ref{Theorem_1}. $\beta$ is the forward SDE with jump diffusions and random coefficients, and from Assumptions \ref{Assumption_1} and \ref{Assumption_2}, it admits a unique solution in $\mathcal{C}_{\mathbb{F}}^2(t,T;\mathbb{R}^n)$. Moreover, $(\alpha,\eta,\gamma)$ is a linear BSDE with jump diffusions and random coefficients, which admits a unique solution in $\mathcal{C}_{\mathbb{F}}^2(t,T;\mathbb{R}^n) \times \mathcal{L}_{\mathbb{F}}^2(t,T;\mathbb{R}^n) \times \mathcal{L}_{\mathbb{F},p}^2(t,T;\mathbb{R}^n)$ \cite[Lemma 2.4]{Tang_SICON_1994} (see also \cite[Theorem 2.1]{Barles_SSR_1997}).

We now consider the duality relation between $(x^\prime,\phi^\prime)$ and $(\alpha,\beta)$ using It\^o's formula (note that $\widehat{B}_2$, $\widetilde{H}_2$ and $\overline{K}_2$ in (\ref{eq_2_10_1_1_1}) are symmetric):
\begin{align}
\label{eq_3_5}
& \dd \langle x^\prime(s), \alpha(s) \rangle 
 = \Big [ \widehat{A}(s) x^\prime(s-) + \widehat{B}_2(s) 	 \phi^\prime(s-) + \widehat{H}_2(s) \theta^\prime(s)  \\
&~~~~~~~~~~ + \int_{E} \widehat{K}_2(s,e) \psi^\prime(s,e) \lambda (\dd e) + \widehat{B}_1(s) u_1^\prime(s) \Big ]^\top \alpha(s-) \dd s \nonumber \\
&~~~ - x^\prime(s-)^\top \Bigl [ \widehat{A}(s)^\top \alpha(s-) + 	\widehat{C}(s)^\top \eta(s)  + \int_{E} \widehat{F}(s,e)^\top \gamma(s,e) \lambda (\dd e) + Q_1(s) \overline{x}(s-)  \Bigr ]\dd s \nonumber \\
&~~~ + \Big [ \widehat{C}(s) x^\prime(s-) + \widehat{H}_2(s)^\top \phi^\prime(s-) + \widetilde{H}_2(s)\theta^\prime(s) \nonumber \\
&~~~~~~~~~~ + \int_{E} \widetilde{K}_2(s,e)\psi^\prime(s,e)\lambda(\dd e) + \widehat{D}_1(s) u_1^\prime(s)\Big ]^\top \eta(s) \dd s + \int_{E} \Big[ \widehat{F}(s,e)  x^\prime(s-)  \nonumber \\
&~~~~~~~~~~ + \widehat{K}_2(s,e)^\top \phi(s-) + \widetilde{K}_2(s,e)^\top \theta^\prime(s) + \int_{E} \overline{K}_2(s,e,e^\prime) \psi^\prime(s,e^\prime) \lambda (\dd e^\prime) \nonumber \\
&~~~~~~~~~~ + \widehat{G}_1(s,e) u_1^\prime(s) \Bigr ]^\top \gamma(s,e) \lambda(\dd e) \dd s  + [\cdots]\dd B(s) + \int_{E} [\cdots] \widetilde{N}(\dd e, \dd s), \nonumber
\end{align}
and
\begin{align}
\label{eq_3_6}
& \dd \langle \phi^\prime(s), \beta(s) \rangle = - \Bigl [  \widehat{A}(s)^\top \phi^\prime(s-)  + \widehat{C}(s)^\top \theta^\prime(s) + \int_{E} \widehat{F}(s,e)^\top \psi^\prime(s,e) \lambda(\dd e)   \\
&~~~ + \widehat{H}_1(s)^\top u_1^\prime(s) + \int_{E} \widehat{K}_1(s,e)^\top u_1^\prime(s) \lambda( \dd e)  \Bigr ]^\top \beta(s-) \dd s   \nonumber \\
&~~~ + \phi^\prime(s-)^\top \Bigl [ \widehat{A}(s)\beta(s-) + \widehat{B}_2(s) \alpha(s-) + \widehat{H}_2(s) \eta(s) + \int_{E} \widehat{K}_2(s,e) \gamma(s,e) \lambda(\dd e) \Bigr ] \dd s \nonumber \\
&~~~ + \theta^\prime(s)^\top \Bigl [ \widehat{C}(s)\beta(s-) + \widehat{H}_2(s)^\top \alpha(s-) + \widetilde{H}_2(s) \eta^\prime(s) + \int_{E} \widetilde{K}_2(s,e) \gamma(s,e) \lambda(\dd e) \Bigr ] \dd s \nonumber \\
&~~~ + \int_{E} \psi^\prime(s,e)^\top \Bigl [ \widehat{F}(s,e)\beta(s-) + \widehat{K}_2(s,e)^\top \alpha(s-) + \widetilde{K}_2(s,e)^\top \eta(s) \nonumber \\
&~~~~~~~~~~ + \int_{E} \overline{K}_2(s,e,e^\prime) \gamma(s,e^\prime) \lambda (\dd e^\prime) \Bigr ] \lambda(\dd e) \dd s   + [\cdots] \dd B(s)  + \int_{E} [\cdots] \widetilde{N}(\dd e, \dd s). \nonumber
\end{align}
Using (\ref{eq_3_5}) and (\ref{eq_3_6}), we have
\begin{align}
\label{eq_3_7}
& \mathbb{E} [\langle x^\prime(T), M_1 \overline{x}(T) \rangle] \\
& = 	\mathbb{E} \bigl [ \langle x^\prime(T), \alpha(T)\rangle -  \langle x^\prime(t),\alpha(t) \rangle - \langle \phi^\prime(T),\beta(T) \rangle +  \langle \phi^\prime(t),\beta(t) \rangle \bigr ] \nonumber \\
& = \mathbb{E} \Bigl [ \int_t^T \bigl [ - \langle x^\prime(s), Q_1(s) \overline{x}(s) \rangle + \Bigl \langle u_1^\prime(s), \widehat{B}_1(s)^\top \alpha(s-) + \widehat{D}_1(s)^\top \eta(s) \nonumber \\
&~~~ + \int_{E} \widehat{G}_1(s)^\top \gamma(s,e) \lambda (\dd e) + \widehat{H}_1(s) \beta(s-) + \int_{E} \widehat{K}_1(s,e) \beta(s-) \lambda (\dd e) \Bigr \rangle \bigr ] \dd s \Bigr ]. \nonumber
\end{align}
Similarly, we have
\begin{align}
\label{eq_3_8}
J_1(a;\overline{u}_1,\overline{u}_2) & = \mathbb{E} \Bigl [ \langle \overline{x}(t),\alpha(t) \rangle + 	\int_t^T \langle \overline{u}_1(s), R_1(s) \overline{u}_1(s) \rangle \\
&~~~ + \Bigl \langle \overline{u}_1(s), \widehat{B}_1(s)^\top \alpha(s-) + \widehat{D}_1(s)^\top \eta(s) + \int_{E} \widehat{G}_1(s)^\top \gamma(s,e) \lambda (\dd e) \nonumber \\
&~~~  + \widehat{H}_1(s) \beta(s-) + \int_{E} \widehat{K}_1(s,e) \beta(s-) \lambda (\dd e) \Bigr \rangle \dd s \Bigr ], \nonumber
\end{align}
and
\begin{align}
\label{eq_3_9}
& J_1(0;u_1^\prime,\overline{u}_2) \\
& =\mathbb{E} \Bigl [ \int_t^T \bigl [ |x^\prime(s)|^2_{Q_1(s)} + |u_1^\prime(s)|^2_{R_1(s)} \bigr ] \dd s + |x^\prime(T)|^2_{M_1} \Bigr ] \nonumber \\
& =  \mathbb{E} \Bigl [  \int_t^T \Bigl \langle u_1^\prime(s), R_1(s) u_1^\prime(s) + \widehat{B}_1(s)^\top \alpha^\prime(s-)  + \widehat{D}_1(s)^\top \eta^\prime(s)  \nonumber  \\
&~~~ + \int_{E} \widehat{G}_1(s)^\top \gamma^\prime(s,e) \lambda (\dd e)  + \widehat{H}_1(s) \beta^\prime(s-)  + \int_{E} \widehat{K}_1(s,e) \beta^\prime(s-) \lambda (\dd e) \Bigr \rangle \dd s \Bigr ].  \nonumber
\end{align}

Then from (\ref{eq_3_7})-(\ref{eq_3_9}), for $\kappa \in \mathbb{R}$, 
\begin{align*}
& J(a;\overline{u}_1 + \kappa u_1^\prime, \overline{u}_2) - 	J(a;\overline{u}_1, \overline{u}_2) \\
& = 2 \kappa \mathbb{E} \Bigl [ \int_t^T \bigl [ \langle x^\prime(s), Q_1(s) \overline{x}(s) \rangle + \langle u_1^\prime(s), R_1(s) \overline{u}_1(s) \rangle \bigr ] \dd s + \langle x^\prime(T), M_1 \overline{x}(T) \rangle \Bigr ] \\
&~~~ + \kappa^2 J_1(0,u_1^\prime,\overline{u}_2) \\
& = 2 \kappa \mathbb{E} \Biggl [ \int_t^T  \Bigl \langle u_1^\prime(s), R_1(s) \overline{u}_1(s) + \widehat{B}_1(s)^\top \alpha(s-) + \widehat{D}_1(s)^\top \eta(s) \\
&~~~~~~~ + \int_{E} \widehat{G}_1(s)^\top \gamma(s,e) \lambda (\dd e) + \widehat{H}_1(s) \beta(s-) + \int_{E} \widehat{K}_1(s,e) \beta(s-) \lambda (\dd e) \Bigr \rangle  \dd s \Biggr ] \\
&~~~ + \kappa^2 \mathbb{E} \Biggl [  \int_t^T \Bigl \langle u_1^\prime(s), R_1(s) u_1^\prime(s) + \widehat{B}_1(s)^\top \alpha^\prime(s-)  + \widehat{D}_1(s)^\top \eta^\prime(s)  \\
&~~~~~~~ + \int_{E} \widehat{G}_1(s)^\top \gamma^\prime(s,e) \lambda (\dd e)   + \widehat{H}_1(s) \beta^\prime(s-) + \int_{E} \widehat{K}_1(s,e) \beta^\prime(s-) \lambda (\dd e) \Bigr \rangle \dd s \Biggr ] \geq 0,
\end{align*}
which implies that under (\ref{eq_3_3}), $\overline{u}_1 \in \mathcal{U}_1$ is the optimal control for \textbf{(LQ-L)} if and only if  the first-order optimality condition in (\ref{eq_3_4}) holds. This completes the proof.
\end{proof}

\begin{remark}\label{Remark_2}
{\color{black}Note that (\ref{eq_3_3}) is needed, since \textbf{(LQ-L)} is the indefinite LQ problem.} From (\ref{eq_3_9}), we can see that (\ref{eq_3_3}) holds when $Q_1$, $R_1$ and $M_1$ are uniformly positive (semi)definite for a.e. $(\omega,s) \in \Omega \times [0,T]$.
\end{remark}

Below, we obtain the state-feedback representation of (\ref{eq_3_4}) for two different cases. Note that (\ref{eq_3_4}) depends on the coupled FBSDEs in (\ref{eq_3_2}).

Let us define the augmented (forward and backward) state processes
\begin{align}
\label{eq_3_10_1_1_1_1_1_1}
\mathcal{X}(s) := \begin{bmatrix}
 	\overline{x}(s) \\
 	\beta(s)
 \end{bmatrix},~ \mathcal{Y}(s) := \begin{bmatrix}
 \alpha(s) \\
 \phi(s)	
 \end{bmatrix},~ \mathcal{Z}(s) := \begin{bmatrix}
	\eta(s) \\
 	\theta(s)
 \end{bmatrix},~ \mathcal{K}(s,e) := \begin{bmatrix}
 	\gamma(s,e) \\
 	\psi(s,e)
 \end{bmatrix},
\end{align}
where $\overline{\mathcal{X}} := \mathcal{X}(t) = \begin{bmatrix}
 a \\ 0	
 \end{bmatrix}$ and we define (see (\ref{eq_2_9_1}), (\ref{eq_2_10_1_1_1}) and (\ref{eq_2_13_1_1_1}))
\begin{align}
\label{eq_3_10_1_1_1_1_1}
\begin{cases}
\mathbb{A}(s) := \begin{bmatrix}
 	\widehat{A}(s) & 0 \\
 	0 & \widehat{A}(s)
 \end{bmatrix},~ \mathbb{B}_2(s) := \begin{bmatrix}
 	0 & \widehat{B}_2(s) \\
 	\widehat{B}_2(s) & 0
 \end{bmatrix},~ \widehat{\mathbb{H}}(s) := \begin{bmatrix}
 	0 & \widehat{H}_2(s) \\
 	\widehat{H}_2(s) & 0
 \end{bmatrix} \\
\widehat{\mathbb{K}}(s,e) := \begin{bmatrix}
 	0 & \widehat{K}_2(s,e) \\
 	\widehat{K}_2(s,e) & 0
 \end{bmatrix},~ \mathbb{B}_1(s) := \begin{bmatrix}
 	\widehat{B}_1(s) \\
 	0
 \end{bmatrix},~ \mathbb{C}(s) :=  \begin{bmatrix}
 \widehat{C}(s) & 0 \\
 0 & \widehat{C}(s)	
 \end{bmatrix} \\
 \widetilde{\mathbb{H}}(s) := \begin{bmatrix}
 	0 & \widetilde{H}_2(s) \\
 	\widetilde{H}_2(s) & 0
 \end{bmatrix},~ \widetilde{\mathbb{K}}(s,e) := \begin{bmatrix}
	0 & \widetilde{K}_2(s,e) \\
	\widetilde{K}_2(s,e) & 0 
\end{bmatrix} \\
\mathbb{D}_1(s) := \begin{bmatrix}
 	\widehat{D}_1(s) \\
 	0
 \end{bmatrix},~ \mathbb{F}(s,e)  := \begin{bmatrix}
 \widehat{F}(s,e) & 0 \\
 0 & \widehat{F}(s,e)	
 \end{bmatrix},~ \mathbb{Q}(s) := \begin{bmatrix}
 	Q_1(s) & 0 \\
 	0 & 0
 \end{bmatrix} \\
 \overline{\mathbb{K}}(s,e,e^\prime) := \begin{bmatrix}
 0 & \overline{K}_2(s,e,e^\prime)	\\
 \overline{K}_2(s,e,e^\prime) & 0
 \end{bmatrix},~ \mathbb{G}_1(s,e) := \begin{bmatrix}
 \widehat{G}_1(s,e) \\
 0	
 \end{bmatrix} \\
 \mathbb{H}_1(s) := \begin{bmatrix}
 	0 & \widehat{H}_1(s)
 \end{bmatrix}, ~ \mathbb{K}_1(s,e) := \begin{bmatrix}
 	0 & \widehat{K}_1(s,e)
 \end{bmatrix},~ \mathbb{M}_1 := \begin{bmatrix}
 	M_1 & 0 \\
 	0 & 0
 \end{bmatrix}.
\end{cases}
\end{align}

Then the (augmented) coupled FBSDEs in (\ref{eq_3_2}) can be rewritten as follows:
\begin{align}
\label{eq_3_10}
\begin{cases}
\dd \mathcal{X}(s) = \Bigl [ \mathbb{A}(s)\mathcal{X}(s-) + \mathbb{B}_2(s) \mathcal{Y}(s-) + \widehat{\mathbb{H}}(s)\mathcal{Z}(s) \\
~~~~~~~~~~ + \int_{E} \widehat{\mathbb{K}}(s,e) \mathcal{K}(s,e) \lambda (\dd e) + \mathbb{B}_1(s) \overline{u}_1(s) \Bigr ]\dd s \\
~~~ + \Bigl [ \mathbb{C}(s)\mathcal{X}(s-) + \widehat{\mathbb{H}}(s)^\top \mathcal{Y}(s-) + \widetilde{\mathbb{H}}(s) \mathcal{Z}(s) \\
~~~~~~~~~~ + \int_{E} \widetilde{\mathbb{K}}(s,e) \mathcal{K}(s,e) \lambda (\dd e) + \mathbb{D}_1(s) \overline{u}_1(s) \Bigr ] \dd B(s) \\
~~~ + \int_{E} \Bigl [ \mathbb{F}(s,e) \mathcal{X}(s-) + \widehat{\mathbb{K}}(s,e)^\top \mathcal{Y}(s-) + \widetilde{\mathbb{K}}(s,e)^\top \mathcal{Z}(s) \\
~~~~~~~~~~ + \int_{E} \overline{\mathbb{K}}(s,e,e^\prime) \mathcal{K}(s,e^\prime) \lambda (\dd e^\prime) + \mathbb{G}_1(s,e) \overline{u}_1(s) \Bigr ]  \widetilde{N}(\dd e, \dd s),~ s \in (t,T] \\
\dd \mathcal{Y}(s) = - \Bigl [ \mathbb{A}(s)^\top \mathcal{Y}(s-) + \mathbb{Q}(s) \mathcal{X}(s-) +   \mathbb{C}(s)^\top \mathcal{Z}(s) \\
~~~~~~~~~~ + \int_{E} \mathbb{F}(s,e)^\top \mathcal{K}(s,e) \lambda (\dd e) + \mathbb{H}_1(s)^\top \overline{u}_1(s)  \\
~~~~~~~~~~ + \int_{E} \mathbb{K}_1(s,e)^\top \overline{u}_1(s) \lambda (\dd e) \Bigr ] \dd s \\
~~~ + \mathcal{Z}(s) \dd B(s) + \int_{E} \mathcal{K}(s,e) \widetilde{N}(\dd e, \dd s),~ s \in [t,T) \\
\mathcal{X}(t) = \overline{\mathcal{X}},~ \mathcal{Y}(T) = \mathbb{M}_1 \mathcal{X}(T),
\end{cases}
\end{align}
where the optimality condition in (\ref{eq_3_4}) becomes
\begin{align}
\label{eq_3_11}
& R_1(s) \overline{u}_1(s) + \mathbb{B}_1(s)^\top \mathcal{Y}(s-) + \mathbb{D}_1(s)^\top \mathcal{Z}(s) + \int_{E} \mathbb{G}_1(s,e)^\top \mathcal{K}(s,e) \lambda (\dd e) \\
&~~~ + \mathbb{H}_1(s) \mathcal{X}(s-) + \int_{E} \mathbb{K}_1(s,e) \mathcal{X}(s-) \lambda (\dd e) = 0. \nonumber
\end{align}

We consider the following transformation in the \emph{Four-Step Scheme}:
\begin{align}
\label{eq_3_12}
\mathcal{Y}(s) = \mathcal{P}(s) \mathcal{X}(s),	
\end{align}
where $\mathcal{P}$ takes the following form:
\begin{align}	
\label{eq_3_13}
\begin{cases}
\dd \mathcal{P}(s) = \Lambda_3(s) \dd s + \Psi(s) \dd B(s) + \int_{E} \Theta(s,e) \widetilde{N}(\dd e, \dd s),~ s \in [t,T) \\
\mathcal{P}(T) = \mathbb{M}_1.
\end{cases}
\end{align}
Note that $\mathcal{P}$, $\Psi$ and $\Theta$ are $\mathbb{R}^{2n \times 2n}$-valued processes. Let ($s$ is suppressed)
\begin{align*}
\mathcal{P} = \begin{bmatrix}
	\mathcal{P}_{11} & \mathcal{P}_{12} \\
	\mathcal{P}_{21} & \mathcal{P}_{22}
\end{bmatrix},~ \text{$\mathcal{P}_{11}$ is an $\mathbb{R}^{n \times n}$-dimensional process.}
\end{align*}

By applying It\^o's formula to (\ref{eq_3_12}) and using (\ref{eq_3_13}), we have
\begin{align}
\label{eq_3_14}
\dd \mathcal{Y}(s) & = - \Bigl [ \mathbb{A}(s)^\top \mathcal{P}(s-) \mathcal{X}(s-) + \mathbb{Q}(s) \mathcal{X}(s-) +   \mathbb{C}(s)^\top \mathcal{Z}(s)  \\
& ~~~~~~~~~~ + \mathbb{H}_1(s)^\top \overline{u}_1(s) + \int_{E} \mathbb{F}(s,e)^\top \mathcal{K}(s,e)    \lambda (\dd e)  \nonumber \\
& ~~~~~~~~~~ + \int_{E} \mathbb{K}_1(s,e)^\top \overline{u}_1(s) \lambda (\dd e) \Bigr ] \dd s \nonumber \\
& ~~~  + \mathcal{Z}(s) \dd B(s) + \int_{E} \mathcal{K}(s,e) \widetilde{N}(\dd e, \dd s) \nonumber \\
& = \Bigl [ \Lambda_3(s) \dd s + \Psi(s) \dd B(s) + \int_{E} \Theta(s,e) \widetilde{N}(\dd e, \dd s) \Bigr ] \mathcal{X}(s-) \nonumber \\
&~~~ + \mathcal{P}(s-) \Bigl [ \mathbb{A}(s)\mathcal{X}(s-) + \mathbb{B}_2(s) \mathcal{P}(s-) \mathcal{X}(s-)  + \widehat{\mathbb{H}}(s)\mathcal{Z}(s) \nonumber \\
& ~~~~~~~~~~ + \int_{E} \widehat{\mathbb{K}}(s,e) \mathcal{K}(s,e) \lambda (\dd e) + \mathbb{B}_1(s) \overline{u}_1(s) \Bigr ]\dd s 	\nonumber \\
&~~~ + \mathcal{P}(s-) \Bigl [ \mathbb{C}(s)\mathcal{X}(s-) + \widehat{\mathbb{H}}(s)^\top\mathcal{P}(s-) \mathcal{X}(s-)  + \widetilde{\mathbb{H}}(s) \mathcal{Z}(s) \nonumber \\
& ~~~~~~~~~~ + \int_{E} \widetilde{\mathbb{K}}(s,e) \mathcal{K}(s,e) \lambda (\dd e) + \mathbb{D}_1(s) \overline{u}_1(s) \Bigr ] \dd B(s) \nonumber \\
&~~~ + \Psi(s)^\top \Bigl [ \mathbb{C}(s)\mathcal{X}(s-) + \widehat{\mathbb{H}}(s)^\top \mathcal{P}(s-) \mathcal{X}(s-)  + \widetilde{\mathbb{H}}(s) \mathcal{Z}(s) \nonumber \\
& ~~~~~~~~~~ + \int_{E} \widetilde{\mathbb{K}}(s,e) \mathcal{K}(s,e) \lambda (\dd e) + \mathbb{D}_1(s) \overline{u}_1(s) \Bigr ] \dd s \nonumber \\
&~~~ + \int_{E} \Theta(s,e)^\top  \Bigl [ \mathbb{F}(s,e) \mathcal{X}(s-) + \widehat{\mathbb{K}}(s,e)^\top \mathcal{P}(s-) \mathcal{X}(s-)  \nonumber \\
& ~~~~~~~~~~ + \widetilde{\mathbb{K}}(s,e)^\top \mathcal{Z}(s) + \int_{E} \overline{\mathbb{K}}(s,e,e^\prime) \mathcal{K}(s,e^\prime) \lambda (\dd e^\prime) \nonumber \\
& ~~~~~~~~~~  + \mathbb{G}_1(s,e) \overline{u}_1(s) \Bigr ] \lambda (\dd e) \dd s \nonumber \\
&~~~ + \int_{E} (\mathcal{P}(s-) + \Theta(s,e))  \Bigl [ \mathbb{F}(s,e) \mathcal{X}(s-) + \widehat{\mathbb{K}}(s,e)^\top \mathcal{P}(s-) \mathcal{X}(s-)  \nonumber \\
& ~~~~~~~~~~  + \widetilde{\mathbb{K}}(s,e)^\top \mathcal{Z}(s)  \nonumber \\
& ~~~~~~~~~~ + \int_{E} \overline{\mathbb{K}}(s,e,e^\prime) \mathcal{K}(s,e^\prime) \lambda (\dd e^\prime) + \mathbb{G}_1(s,e) \overline{u}_1(s) \Bigr ] \widetilde{N}(\dd e, \dd s). \nonumber
\end{align}

To obtain the state-feedback representation of (\ref{eq_3_11}), we consider the following two different cases:
\begin{enumerate}[(i)]
	\item The Poisson process $N$ has jumps of unit size ($E=\{1\}$);
	\item The follower's control is not included in the jump part of (\ref{eq_1}) ($G_2 = 0$).
\end{enumerate}
\begin{remark}
A detailed discussion on these two assumptions is given in Remark \ref{Remark_7}.	
\end{remark}

\subsection{Case I: $N$ has jumps of unit size}\label{Section_3_1}
Let us assume that
\begin{assumption}\label{Assumption_3}
The Poisson process $N$ has jumps of unit size, i.e., $E = \{ 1 \}$. 
\end{assumption}
We continue the Four-Step Scheme under Assumption \ref{Assumption_3}.

Under Assumption \ref{Assumption_3} and from Remark \ref{Remark_1}, (\ref{eq_3_14}) is given by\footnote{Note that under Assumption \ref{Assumption_3}, $\int_{E} g(s,e) \lambda (\dd e) \dd s = g(s) \lambda \dd s$ for $g \in \mathcal{G}_{\mathbb{F},p}^2(t,T,\lambda;\mathbb{R}^n)$, where $\lambda > 0$ is the intensity of $N$ \cite{Applebaum_book, Privault_book}.}

\begin{align}
\label{eq_3_15}
\dd \mathcal{Y}(s) & = - \Bigl [ \mathbb{A}(s)^\top \mathcal{P}(s) \mathcal{X}(s-) + \mathbb{Q}(s) \mathcal{X}(s-) +   \mathbb{C}(s)^\top \mathcal{Z}(s)  + \mathbb{H}_1(s)^\top \overline{u}_1(s) \\
& ~~~~~~~~~~ + \lambda \mathbb{F}(s)^\top \mathcal{K}(s) + \lambda \mathbb{K}_1(s)^\top \overline{u}_1(s)  \Bigr ] \dd s  + \mathcal{Z}(s) \dd B(s) + \mathcal{K}(s) \dd \widetilde{N}(s) \nonumber \\
& = \Bigl [ \Lambda_3(s) \dd s + \Psi(s) \dd B(s) +  \Theta(s) \dd \widetilde{N}(s) \Bigr ] \mathcal{X}(s-) \nonumber \\
&~~~ + \mathcal{P}(s-) \Bigl [ \mathbb{A}(s)\mathcal{X}(s-) + \mathbb{B}_2(s) \mathcal{P}(s-) \mathcal{X}(s-)  + \widehat{\mathbb{H}}(s)\mathcal{Z}(s) \nonumber \\
& ~~~~~~~~~~ + \lambda \widehat{\mathbb{K}}(s) \mathcal{K}(s) + \mathbb{B}_1(s) \overline{u}_1(s) \Bigr ]\dd s 	\nonumber \\
&~~~ + \mathcal{P}(s-) \Bigl [ \mathbb{C}(s)\mathcal{X}(s-) + \widehat{\mathbb{H}}(s)^\top\mathcal{P}(s-) \mathcal{X}(s-)  + \widetilde{\mathbb{H}}(s) \mathcal{Z}(s) \nonumber \\
& ~~~~~~~~~~ +\lambda \widetilde{\mathbb{K}}(s) \mathcal{K}(s) + \mathbb{D}_1(s) \overline{u}_1(s) \Bigr ] \dd B(s) \nonumber \\
&~~~ + \Psi(s)^\top \Bigl [ \mathbb{C}(s)\mathcal{X}(s-) + \widehat{\mathbb{H}}(s)^\top \mathcal{P}(s-) \mathcal{X}(s-)  + \widetilde{\mathbb{H}}(s) \mathcal{Z}(s) \nonumber \\
& ~~~~~~~~~~ + \lambda \widetilde{\mathbb{K}}(s) \mathcal{K}(s) + \mathbb{D}_1(s) \overline{u}_1(s) \Bigr ] \dd s \nonumber \\
&~~~ + \Theta(s)^\top  \Bigl [ \mathbb{F}(s) \mathcal{X}(s-) + \widehat{\mathbb{K}}(s)^\top \mathcal{P}(s-) \mathcal{X}(s-)  + \widetilde{\mathbb{K}}(s)^\top \mathcal{Z}(s) \nonumber \\
& ~~~~~~~~~~ + \lambda \overline{\mathbb{K}}(s) \mathcal{K}(s)  + \mathbb{G}_1(s) \overline{u}_1(s) \Bigr ] \lambda \dd s \nonumber \\
&~~~ + (\mathcal{P}(s-) + \Theta(s))  \Bigl [ \mathbb{F}(s) \mathcal{X}(s-) + \widehat{\mathbb{K}}(s)^\top \mathcal{P}(s-) \mathcal{X}(s-)  + \widetilde{\mathbb{K}}(s)^\top \mathcal{Z}(s) \nonumber \\
& ~~~~~~~~~~ + \lambda \overline{\mathbb{K}}(s) \mathcal{K}(s)  + \mathbb{G}_1(s) \overline{u}_1(s) \Bigr ] \dd \widetilde{N}(s). \nonumber
\end{align}

Let us define ($s$ is suppressed)
\begin{align}
\label{eq_3_17_1_1_1_1_1}
\begin{cases}
\mathscr{A}_{11} :=  I - \mathcal{P}(s-) \widetilde{\mathbb{H}},~ \mathscr{A}_{12} := - \lambda \mathcal{P}(s-) \widetilde{\mathbb{K}}\\
\mathscr{A}_{21} := - (\mathcal{P}(s-) + \Theta) \widetilde{\mathbb{K}}^\top,~ \mathscr{A}_{22} := I - \lambda (\mathcal{P}(s-) + \Theta)\overline{\mathbb{K}} \\
\mathscr{B}_{11} := \mathcal{P}(s-) \mathbb{C} + \mathcal{P}(s-) \widehat{\mathbb{H}}^\top \mathcal{P}(s-) + \Psi,~ \mathscr{B}_{12} := \mathcal{P}(s-) \mathbb{D}_1 \\
\mathscr{B}_{21} := (\mathcal{P}(s-) + \Theta)(\mathbb{F} + \widehat{\mathbb{K}}^\top \mathcal{P}(s-)) + \Theta,~ \mathscr{B}_{22} := (\mathcal{P}(s-) + \Theta) \mathbb{G}_1.
\end{cases}
\end{align}
Then from (\ref{eq_3_15}), we can see that
\begin{align}
\label{eq_3_16_1}
& \mathscr{A} \begin{bmatrix}
	\mathcal{Z} \\
	\mathcal{K}
\end{bmatrix} = \begin{bmatrix}
	\mathscr{A}_{11} & \mathscr{A}_{12} \\
	\mathscr{A}_{21} & \mathscr{A}_{22}
\end{bmatrix}
\begin{bmatrix}
	\mathcal{Z} \\
	\mathcal{K}
\end{bmatrix}  = \begin{bmatrix}
\mathscr{B}_{11} \mathcal{X}(s-) + \mathscr{B}_{12} \overline{u}_1 \\
\mathscr{B}_{21} \mathcal{X}(s-) + \mathscr{B}_{22} \overline{u}_1 \\
 \end{bmatrix},
\end{align}
which, together with the block matrix inversion lemma  \cite[page 18]{Horn_book_2013} (assuming its invertibility), implies
\begin{align}
\label{eq_3_16}	
\begin{bmatrix}
	\mathcal{Z} \\
	\mathcal{K}
\end{bmatrix} = \mathscr{A}^{-1} \begin{bmatrix}
\mathscr{B}_{11} \mathcal{X}(s-) + \mathscr{B}_{12} \overline{u}_1 \\
\mathscr{B}_{21} \mathcal{X}(s-) + \mathscr{B}_{22} \overline{u}_1 \\
 \end{bmatrix} = \begin{bmatrix}
 	\widehat{\mathscr{A}}_{11} & \widehat{\mathscr{A}}_{12} \\
 	\widehat{\mathscr{A}}_{21} & \widehat{\mathscr{A}}_{22}
 \end{bmatrix} \begin{bmatrix}
\mathscr{B}_{11} \mathcal{X}(s-) + \mathscr{B}_{12} \overline{u}_1 \\
\mathscr{B}_{21} \mathcal{X}(s-) + \mathscr{B}_{22} \overline{u}_1 \\
 \end{bmatrix},
\end{align}
where
\begin{align}
\label{eq_3_20_1_1_1_1}
\begin{cases}	
\widehat{\mathscr{A}}_{11}  := (\mathscr{A}_{11} - \mathscr{A}_{12} \mathscr{A}_{22}^{-1} \mathscr{A}_{21})^{-1},~ \widehat{\mathscr{A}}_{12} := \mathscr{A}_{11}^{-1} \mathscr{A}_{12} (\mathscr{A}_{21} \mathscr{A}_{11}^{-1} \mathscr{A}_{12} - \mathscr{A}_{22})^{-1} \\
\widehat{\mathscr{A}}_{21} := \mathscr{A}_{22}^{-1}\mathscr{A}_{21} (\mathscr{A}_{12}\mathscr{A}_{22}^{-1} \mathscr{A}_{21} - \mathscr{A}_{11})^{-1}, ~\widehat{\mathscr{A}}_{22} := (\mathscr{A}_{22} - \mathscr{A}_{21} \mathscr{A}_{11}^{-1} \mathscr{A}_{12})^{-1}.
\end{cases}
\end{align}

Substituting (\ref{eq_3_16}) and (\ref{eq_3_12}) into the optimality condition in (\ref{eq_3_11}) yields\footnote{Under Assumption \ref{Assumption_3}, (\ref{eq_3_11}) becomes $ R_1(s) \overline{u}_1(s) + \mathbb{B}_1(s)^\top \mathcal{Y}(s-) + \mathbb{D}_1(s)^\top \mathcal{Z}(s) + \lambda \mathbb{G}_1(s)^\top \mathcal{K}(s) + \mathbb{H}_1(s) \mathcal{X}(s-) + \lambda \mathbb{K}_1(s) \mathcal{X}(s-) = 0$.}
\begin{align*}
& R_1 \overline{u}_1(s) + \mathbb{B}_1^\top \mathcal{P}(s-) \mathcal{X}(s-) + \mathbb{H}_1 \mathcal{X}(s-) + \lambda \mathbb{K}_1 \mathcal{X}(s-)  \\
&~~~ + \bigl ( \mathbb{D}_1^\top (\widehat{\mathscr{A}}_{11} \mathscr{B}_{11} + \widehat{\mathscr{A}}_{12} \mathscr{B}_{21}) + \lambda \mathbb{G}_1^\top (\widehat{\mathscr{A}}_{21} \mathscr{B}_{11} + \widehat{\mathscr{A}}_{22} \mathscr{B}_{21}) \bigr ) \mathcal{X}(s-) \\
&~~~ + \bigl ( \mathbb{D}_1^\top (\widehat{\mathscr{A}}_{11} \mathscr{B}_{12} + \widehat{\mathscr{A}}_{12} \mathscr{B}_{22}) + \lambda \mathbb{G}_1^\top (\widehat{\mathscr{A}}_{21} \mathscr{B}_{12} + \widehat{\mathscr{A}}_{22} \mathscr{B}_{22}) \bigr ) \overline{u}_1(s) = 0,
\end{align*}
and we have
\begin{align}
\label{eq_3_17}
\overline{u}_1(s) 
& = - \mathcal{R}_1(s)^{-1} \mathcal{H}_1(s) \mathcal{X}(s-),
\end{align}
provided that $\mathcal{R}_1$ is invertible, where ($s$ is suppressed)
\begin{align}
\label{eq_3_19}
\begin{cases}
\mathcal{R}_1 := R_1 + 	\bigl ( \mathbb{D}_1^\top (\widehat{\mathscr{A}}_{11} \mathscr{B}_{12} + \widehat{\mathscr{A}}_{12} \mathscr{B}_{22}) + \lambda \mathbb{G}_1^\top (\widehat{\mathscr{A}}_{21} \mathscr{B}_{12} + \widehat{\mathscr{A}}_{22} \mathscr{B}_{22}) \bigr ) \\
\mathcal{H}_1 := \mathbb{B}_1^\top \mathcal{P}(s-) + \mathbb{H}_1 + \lambda \mathbb{K}_1 \\
~~~ + \bigl ( \mathbb{D}_1^\top (\widehat{\mathscr{A}}_{11} \mathscr{B}_{11} + \widehat{\mathscr{A}}_{12} \mathscr{B}_{21}) + \lambda \mathbb{G}_1^\top (\widehat{\mathscr{A}}_{21} \mathscr{B}_{11} + \widehat{\mathscr{A}}_{22} \mathscr{B}_{21}) \bigr ).
\end{cases}
\end{align}

By substituting (\ref{eq_3_17}) into (\ref{eq_3_16}), we have ($s$ is suppressed)
\begin{align}
\label{eq_3_20}
\begin{bmatrix}
\mathcal{Z}(s) \\
\mathcal{K}(s)
\end{bmatrix} = 	\begin{bmatrix}
 		\mathscr{F}_{11} - \mathscr{F}_{12} \mathcal{R}_1^{-1} \mathcal{H}_1\\
 		\mathscr{F}_{21} - \mathscr{F}_{22} \mathcal{R}_1^{-1} \mathcal{H}_1
 	\end{bmatrix}\mathcal{X}(s-),
\end{align}
where
\begin{align}
\label{eq_3_24_1_1_1_1}
\begin{cases}
\mathscr{F}_{11} :=	 \widehat{\mathscr{A}}_{11} \mathscr{B}_{11} + \widehat{\mathscr{A}}_{12} \mathscr{B}_{21},~ \mathscr{F}_{12} :=	 \widehat{\mathscr{A}}_{11} \mathscr{B}_{12} + \widehat{\mathscr{A}}_{12} \mathscr{B}_{22} \\
\mathscr{F}_{21} :=	 \widehat{\mathscr{A}}_{21} \mathscr{B}_{11} + \widehat{\mathscr{A}}_{22} \mathscr{B}_{21},~ \mathscr{F}_{22} :=	 \widehat{\mathscr{A}}_{21} \mathscr{B}_{12} + \widehat{\mathscr{A}}_{22} \mathscr{B}_{22}.
\end{cases}
\end{align}

We substitute (\ref{eq_3_20}) and (\ref{eq_3_17}) into (\ref{eq_3_15}). Then combining (\ref{eq_3_13}) with the above invertibility conditions (see (\ref{eq_3_16}) and (\ref{eq_3_17})) and using the notation in (\ref{eq_3_10_1_1_1_1_1}), (\ref{eq_3_17_1_1_1_1_1}), (\ref{eq_3_20_1_1_1_1}) and (\ref{eq_3_24_1_1_1_1}), the ISRDE in (\ref{eq_3_13}) can be written as
\begin{align}
\label{eq_3_21}
\begin{cases}
\dd \mathcal{P}(s) = - \Bigl [ \mathbb{A}^\top \mathcal{P}(s-) + \mathcal{P}(s-) \mathbb{A} + \mathbb{Q} + \mathcal{P}(s-) \mathbb{B}_2 \mathcal{P}(s-) \\
~~~~~~~~~~~~~ + \Psi^\top \mathbb{C} + \Psi^\top \widehat{\mathbb{H}}^\top \mathcal{P}(s-) + \lambda \Theta^\top \mathbb{F}(s) + \lambda \Theta^\top \widehat{\mathbb{K}}^\top \mathcal{P}(s-) \\
~~~~~~~~~~~~~ + ( \mathbb{C}^\top + \mathcal{P}(s-) \widehat{\mathbb{H}} + \Psi^\top \widetilde{\mathbb{H}} + \lambda \Theta^\top \widetilde{\mathbb{K}}^\top) (\mathscr{F}_{11} - \mathscr{F}_{12} \mathcal{R}_1^{-1} \mathcal{H}_1) \\
~~~~~~~~~~~~~ + (\lambda \mathbb{F}^\top + \lambda \mathcal{P}(s-) \widehat{\mathbb{K}} + \lambda \Psi^\top \widetilde{\mathbb{K}} + \lambda^2 \Theta^\top \overline{\mathbb{K}})(\mathscr{F}_{21} - \mathscr{F}_{22}\mathcal{R}_1^{-1} \mathcal{H}_1)
\\
~~~~~~~~~~~~~ - ( \mathbb{H}_1^\top + \lambda \mathbb{K}_1^\top + \mathcal{P}(s-) \mathbb{B}_1 + \Psi^\top \mathbb{D}_1 + \lambda \Theta^\top \mathbb{G}_1) \mathcal{R}_1^{-1} \mathcal{H}_1 \Bigr ] \dd s \\
~~~~~~~~~ + \Psi(s) \dd B(s) + \Theta(s) \dd \widetilde{N}(s),~ s \in [t,T)\\
\mathcal{P}(T) = \mathbb{M}_1 \\
\det (\mathscr{A}_{11}(s)  ) \neq 0,~ \det (\mathscr{A}_{22}(s)  ) \neq 0,~ \forall s \in [t,T] \\
\det (\mathscr{A}_{11}(s) - \mathscr{A}_{12}(s) \mathscr{A}_{22}(s)^{-1} \mathscr{A}_{21}(s)) \neq 0,~ \forall s \in [t,T] \\
\det (\mathscr{A}_{22}(s) - \mathscr{A}_{21}(s) \mathscr{A}_{11}(s)^{-1} \mathscr{A}_{12}(s)) \neq 0,~ \forall s \in [t,T] \\
\det (\mathcal{R}_1(s)) \neq 0,~ \forall s \in [t,T].
\end{cases}
\end{align}
\begin{remark}\label{Remark_6_1_1_1_1_1_1}
	{\color{black} Unlike \cite{Jiongmin_Yong_1}, note that there are several invertibility conditions of the block matrices to get the explicit expression of the ISRDE in (\ref{eq_3_21}) and the associated (state-feedback type) optimal solution of \textbf{(LQ-L)} in (\ref{eq_3_17}). For the problem without jumps, (\ref{eq_3_16_1}) is simplified to $\mathscr{A}_{12} = \mathscr{A}_{21} = \mathscr{A}_{22} = 0$ as in \cite{Jiongmin_Yong_1}, in which case the ISRDE in (\ref{eq_3_21}) is reduced to the SRDE in \cite[(3.38)]{Jiongmin_Yong_1}.	Hence, the Four-Step Scheme of \textbf{(LQ-L)} under Assumption \ref{Assumption_3} extends \cite[Section 3]{Jiongmin_Yong_1} to the case of jump-diffusion systems.}
\end{remark}

Finally, we substitute (\ref{eq_3_17}), (\ref{eq_3_12}) and  (\ref{eq_3_20}) into $\mathcal{X}$ in (\ref{eq_3_10}). Then
\begin{align}
\label{eq_3_22}
\begin{cases}
\dd \mathcal{X}(s) = \widehat{\mathbb{A}}(s) \mathcal{X}(s-) \dd s + \widehat{\mathbb{C}}(s) \mathcal{X}(s-) \dd B (s) + \widehat{\mathbb{F}}(s) \mathcal{X}(s-) \dd \widetilde{N}(s),~ s \in (t,T] \\
\mathcal{X}(t) = \overline{\mathcal{X}},
\end{cases}	
\end{align}
where ($s$ is suppressed)
\begin{align*}
\begin{cases}
\widehat{\mathbb{A}}	  := \mathbb{A} + \mathbb{B}_2 \mathcal{P}(s-) + \widehat{\mathbb{H}}(\mathscr{F}_{11} - \mathscr{F}_{12} \mathcal{R}_1^{-1} \mathcal{H}_1) \\
~~~~~~~~~~ + \lambda \widehat{\mathbb{K}}(\mathscr{F}_{21} - \mathscr{F}_{22} \mathcal{R}_1^{-1} \mathcal{H}_1) - \mathbb{B}_1 \mathcal{R}_1^{-1} \mathcal{H}_1 \\
\widehat{\mathbb{C}} :=\mathbb{C} + \widehat{\mathbb{H}}^\top \mathcal{P}(s-) + \widetilde{\mathbb{H}} (\mathscr{F}_{11} - \mathscr{F}_{12} \mathcal{R}_1^{-1} \mathcal{H}_1) \\
~~~~~~~~~~ + \lambda \widetilde{\mathbb{K}}(\mathscr{F}_{21} - \mathscr{F}_{22} \mathcal{R}_1^{-1} \mathcal{H}_1) - \mathbb{D}_1 \mathcal{R}_1^{-1} \mathcal{H}_1 \\
\widehat{\mathbb{F}} := \mathbb{F} + \widehat{\mathbb{K}}^\top \mathcal{P}(s-) + \widetilde{\mathbb{K}}^\top (\mathscr{F}_{11} - \mathscr{F}_{12} \mathcal{R}_1^{-1} \mathcal{H}_1) \\
~~~~~~~~~~ + \lambda \overline{\mathbb{K}} (\mathscr{F}_{21} - \mathscr{F}_{22} \mathcal{R}_1^{-1} \mathcal{H}_1) - \mathbb{G}_1 \mathcal{R}_1^{-1} \mathcal{H}_1.
\end{cases}
\end{align*}

In summary, we have the following result:
\begin{theorem}\label{Theorem_2}
Suppose that Assumptions \ref{Assumption_1}-\ref{Assumption_3} hold. Assume that $(\mathcal{P},\Psi,\Theta) \in \mathcal{C}_{\mathbb{F}}^2(t,T;\mathbb{R}^{2n \times 2n}) \times \mathcal{L}_{\mathbb{F}}^2(t,T;\mathbb{R}^{2n \times 2n}) \times \mathcal{L}_{\mathbb{F},p}^2(t,T;\mathbb{R}^{2n \times 2n})$ is the solution of the ISRDE in (\ref{eq_3_21}), and $\mathcal{X}$ is the solution of (\ref{eq_3_22}). Define the transformations in (\ref{eq_3_12}) and (\ref{eq_3_20}), and consider the control in (\ref{eq_3_17}). Then (\ref{eq_3_10}) and (\ref{eq_3_11}) hold. In addition, suppose that (\ref{eq_3_3}) holds. Then the state-feedback type control in (\ref{eq_3_17}) is the optimal control for \textbf{(LQ-L)}, and the associated optimal cost is given by
\begin{align}
\label{eq_3_23}
J_1(a;\overline{u}_1,\overline{u}_2) = \inf_{u_1 \in \mathcal{U}_1} J_1(a;u_1,\overline{u}_2) = \langle a, \mathcal{P}_{11} (t) a \rangle.	
\end{align}
\end{theorem}
\begin{proof}
The statement that (\ref{eq_3_10})-(\ref{eq_3_11}) are equivalent to (\ref{eq_3_12}),  (\ref{eq_3_20}) and (\ref{eq_3_17}) follows from the preceding analysis. Furthermore, Lemma \ref{Lemma_1} implies that under (\ref{eq_3_3}), the state-feedback type control in (\ref{eq_3_17}) is the optimal control for \textbf{(LQ-L)}. 

We now prove (\ref{eq_3_23}). By applying It\^o's formula to (\ref{eq_3_10}) (see (\ref{eq_3_8})), 
\begin{align*}
J_1(a;\overline{u}_1,\overline{u}_2)  & = \mathbb{E} \Bigl [ \langle \mathcal{X}(t),\mathcal{Y}(t) \rangle + 	\int_t^T \Bigl \langle \overline{u}_1(s), R_1(s) \overline{u}_1(s) + \mathbb{B}_1(s)^\top \mathcal{Y}(s-)  \\
&~~~~~ + \mathbb{D}_1(s)^\top \mathcal{Z}(s) + \lambda \mathbb{G}_1(s)^\top \mathcal{K}(s) + \mathbb{H}_1(s) \mathcal{X}(s-) + \lambda \mathbb{K}_1(s) \mathcal{X}(s-) \Bigr \rangle \dd s \Bigr ] \nonumber \\
& = \langle a, \mathcal{P}_{11}(t) a \rangle, \nonumber
\end{align*}
where the second equality follows from (\ref{eq_3_12}), the first-order optimality condition in (\ref{eq_3_11}), and the initial condition $\overline{\mathcal{X}}$. This completes the proof of the theorem.
\end{proof}

Under Assumptions \ref{Assumption_1}-\ref{Assumption_3}, and using (\ref{eq_2_13}) and (\ref{eq_3_17}), we consider
\begin{align}
\label{eq_3_25}
\begin{cases}	
\overline{u}_1(s) = - \mathcal{R}_1(s)^{-1} \mathcal{H}_1(s) \begin{bmatrix}
 x(s-) \\
 \beta(s-)	
 \end{bmatrix} \\
\overline{u}_2(s)  = - \widehat{R}_2(s)^{-1} \widehat{S}_2(s)^\top x(s-) - \widehat{R}_2(s)^{-1} \Bigl (B_2(s)^\top \phi(s-)  \\
~~~ + D_2(s)^\top \theta(s) + \lambda G_2(s)^\top \psi(s) - \widehat{S}_1(s) \mathcal{R}_1(s)^{-1} \mathcal{H}_1(s) \begin{bmatrix}
 x(s-) \\
 \beta(s-)	
 \end{bmatrix}
 \Bigr ).
\end{cases}
\end{align}
Note that $\overline{u}_2$ in (\ref{eq_3_25}) is the state-feedback type optimal control of the follower when $u_1 \equiv \overline{u}_1$. This corresponds to the situation when the leader announces $\overline{u}_1$ to the follower in the Stackelberg game.

\begin{corollary}\label{Corollary_1}
Suppose that the assumptions of Theorems \ref{Theorem_1} and \ref{Theorem_2} hold. Then $(\overline{u}_1,\overline{u}_2) \in \mathcal{U}_1 \times \mathcal{U}_2$ in (\ref{eq_3_25}) constitutes the state-feedback representation of the open-loop Stackelberg equilibrium for the leader and the follower.
\end{corollary}

\subsection{Case II: The jump part in (\ref{eq_1}) does not depend on $u_2$} \label{Section_3_2}

We assume that the control of the follower, $u_2$, is not included in the jump part of (\ref{eq_1}), i.e.,
\begin{assumption}\label{Assumption_4}
$G_2 = 0$.	
\end{assumption}

\begin{remark}\label{Remark_4}
Assumption \ref{Assumption_4} implies that $\widehat{K}_2 = \widetilde{K}_2 = \overline{K}_2 = \widehat{\mathbb{K}} = \widetilde{\mathbb{K}} = \overline{\mathbb{K}} = 0$ $\widehat{F} = F$, and $\widehat{G}_1 = G_1$ (see (\ref{eq_2_10_1_1_1}) and (\ref{eq_3_10_1_1_1_1_1})).	
\end{remark}

Under Assumption \ref{Assumption_4}, in the Four-Step Scheme, (\ref{eq_3_14}) becomes
\begin{align}
\label{eq_3_26}
\dd \mathcal{Y}(s) & = - \Bigl [ \mathbb{A}(s)^\top \mathcal{P}(s-) \mathcal{X}(s-) + \mathbb{Q}(s) \mathcal{X}(s-) +   \mathbb{C}(s)^\top \mathcal{Z}(s)  + \mathbb{H}_1(s)^\top \overline{u}_1(s) \\
& ~~~~~~~~~~ + \int_{E} \mathbb{F}(s,e)^\top \mathcal{K}(s,e)    \lambda (\dd e) + \int_{E} \mathbb{K}_1(s,e)^\top \overline{u}_1(s) \lambda (\dd e) \Bigr ] \dd s \nonumber \\
& ~~~  + \mathcal{Z}(s) \dd B(s) + \int_{E} \mathcal{K}(s,e) \widetilde{N}(\dd e, \dd s) \nonumber \\
& = \Bigl [ \Lambda_3(s) \dd s + \Psi(s) \dd B(s) + \int_{E} \Theta(s,e) \widetilde{N}(\dd e, \dd s) \Bigr ] \mathcal{X}(s-) \nonumber \\
&~~~ + \mathcal{P}(s-) \Bigl [ \mathbb{A}(s)\mathcal{X}(s-) + \mathbb{B}_2(s) \mathcal{P}(s-) \mathcal{X}(s-)  + \widehat{\mathbb{H}}(s)\mathcal{Z}(s) + \mathbb{B}_1(s) \overline{u}_1(s) \Bigr ]\dd s 	\nonumber \\
&~~~ + \mathcal{P}(s-) \Bigl [ \mathbb{C}(s)\mathcal{X}(s-) + \widehat{\mathbb{H}}(s)^\top\mathcal{P}(s-) \mathcal{X}(s-)  + \widetilde{\mathbb{H}}(s) \mathcal{Z}(s) + \mathbb{D}_1(s) \overline{u}_1(s) \Bigr ] \dd B(s) \nonumber \\
&~~~ + \Psi(s)^\top \Bigl [ \mathbb{C}(s)\mathcal{X}(s-) + \widehat{\mathbb{H}}(s)^\top \mathcal{P}(s-) \mathcal{X}(s-)  + \widetilde{\mathbb{H}}(s) \mathcal{Z}(s) + \mathbb{D}_1(s) \overline{u}_1(s) \Bigr ] \dd s \nonumber \\
&~~~ + \int_{E} \Theta(s,e)^\top  \Bigl [ \mathbb{F}(s,e) \mathcal{X}(s-)  + \mathbb{G}_1(s) \overline{u}_1(s) \Bigr ] \lambda (\dd e) \dd s \nonumber \\
&~~~ + \int_{E} (\mathcal{P}(s-) + \Theta(s,e))  \Bigl [ \mathbb{F}(s,e) \mathcal{X}(s-)  + \mathbb{G}_1(s) \overline{u}_1(s) \Bigr ] \widetilde{N}(\dd e, \dd s). \nonumber
\end{align}

With the invertibility of $(I - \mathcal{P}(s-) \widetilde{\mathbb{H}})$, ($s$ is suppressed)
\begin{align}
\label{eq_3_27}
\begin{cases}
	 \mathcal{Z}(s) = (I - \mathcal{P}(s-) \widetilde{\mathbb{H}})^{-1} \bigl ( (\mathcal{P}(s-) \mathbb{C} + \mathcal{P}(s-)  \widehat{\mathbb{H}}^\top \mathcal{P}(s-) + \Psi)\mathcal{X} + \mathcal{P}(s-) \mathbb{D}_1 \overline{u}_1 \bigr )\\
	\mathcal{K}(s,e) = (\Theta(s,e)  + (\mathcal{P}(s-) + \Theta(s,e))\mathbb{F})\mathcal{X} + (\mathcal{P}(s-) + \Theta(s,e)) \mathbb{G}_1 \overline{u}_1.
\end{cases}	
\end{align}

By substituting (\ref{eq_3_27}) into (\ref{eq_3_11}), we have 
\begin{align}
\label{eq_3_28}
\overline{u}_1(s) 
& = - \widehat{\mathcal{R}}_1(s)^{-1} \widehat{\mathcal{B}}_1(s) \mathcal{X}(s-),	 
\end{align}
provided that $\widehat{\mathcal{R}}_1$ is invertible, where ($s$ is suppressed)
\begin{align}
\label{eq_3_29}
\begin{cases}
\widehat{\mathcal{R}}_1 := R_1 + \mathbb{D}_1^\top (I - \mathcal{P}(s-) \widetilde{\mathbb{H}})^{-1} \mathcal{P}(s-) \mathbb{D}_1 \\
~~~ + \int_{E} \mathbb{G}_1(s,e)^\top (\mathcal{P}(s-) + \Theta(s,e)) \mathbb{G}_1(s,e)   \lambda (\dd e) \\
\widehat{\mathcal{B}}_1 = \mathbb{B}_1^\top \mathcal{P}(s-) + \mathbb{H}_1 + \int_{E} \mathbb{K}_1(s,e) \lambda (\dd e) \\
~~~ + \mathbb{D}_1^\top (I - \mathcal{P}(s-) \widetilde{\mathbb{H}})^{-1} (\mathcal{P}(s-) \mathbb{C} + \mathcal{P}(s-)  \widehat{\mathbb{H}}^\top \mathcal{P}(s-) + \Psi)  \\
~~~ + \int_{E} \mathbb{G}_1(s,e)^\top (\Theta(s,e)  + (\mathcal{P}(s-) + \Theta(s,e))\mathbb{F}(s,e)) \lambda (\dd e).
\end{cases}
\end{align}
Then substituting (\ref{eq_3_28}) into (\ref{eq_3_27}) yields
\begin{align}
\label{eq_3_30}
	\mathcal{Z}(s) = \widehat{\mathscr{F}}_{1}(s) \mathcal{X}(s-),~ \mathcal{K}(s,e) = \widehat{\mathscr{F}}_{2}(s,e) \mathcal{X}(s-),
\end{align}
where ($s$ is suppressed)
\begin{align}
\label{eq_3_35_1_1_1_1}
\begin{cases}
\widehat{\mathscr{F}}_{1} := (I - \mathcal{P}(s-) \widetilde{\mathbb{H}})^{-1} \bigl ( (\mathcal{P}(s-) \mathbb{C} + \mathcal{P}(s-)  \widehat{\mathbb{H}}^\top \mathcal{P}(s-) + \Psi) - \mathcal{P}(s-) \mathbb{D}_1 \widehat{\mathcal{R}}_1^{-1} \widehat{\mathcal{B}}_1 \bigr ) \\
\widehat{\mathscr{F}}_{2} := \Theta(s,e)  + (\mathcal{P}(s-) + \Theta(s,e))\mathbb{F} - (\mathcal{P}(s-) + \Theta(s,e)) \mathbb{G}_1 \widehat{\mathcal{R}}_1^{-1} \widehat{\mathcal{B}}_1. 
\end{cases}	
\end{align}
%
%

We substitute (\ref{eq_3_30}) and (\ref{eq_3_28}) into (\ref{eq_3_26}). Then, together with the invertibility conditions in (\ref{eq_3_30}) and (\ref{eq_3_28}) and the notation in (\ref{eq_3_10_1_1_1_1_1}) and (\ref{eq_3_35_1_1_1_1}), the ISRDE in (\ref{eq_3_13}) has to be as follows ($s$ is suppressed):
\begin{align}
\label{eq_3_31}
\begin{cases}
\dd \mathcal{P}(s) = - \Bigl [ \mathbb{A}^\top \mathcal{P}(s-)  + \mathcal{P}(s-) \mathbb{A} + \mathbb{Q} + \mathcal{P}(s-) \mathbb{B}_2 \mathcal{P}(s-) \\
~~~~~~~~~~~~~ + \Psi^\top \mathbb{C} + \Psi \widehat{\mathbb{H}}^\top \mathcal{P}(s-) + \int_{E} \Theta(s,e)^\top \mathbb{F}(s,e) \lambda (\dd e) \\
~~~~~~~~~~~~~ + (\mathbb{C}^\top + \mathcal{P}(s-) \widehat{\mathbb{H}} + \Psi^\top \widetilde{\mathbb{H}}) \widehat{\mathscr{F}}_{1}  + \int_{E} \mathbb{F}(s,e)^\top \widehat{\mathscr{F}}_{2}(s,e) 
 \lambda (\dd e) \\
 ~~~~~~~~~~~~~ - \bigl (\mathbb{H}_1^\top + \int_{E} \mathbb{K}_1(s,e)^\top \lambda (\dd e)   + \mathcal{P}(s-) \mathbb{B}_1 \\
 ~~~~~~~~~~~~~~~~~~ + \Psi^\top \mathbb{D}_1 + \int_{E} \Theta(s,e)^\top \lambda (\dd e) \mathbb{G}_1 \bigr ) \widehat{\mathcal{R}}_{1}^{-1} \widehat{\mathcal{B}}_1 \Bigr ] \dd s\\
~~~~~~~ + \Psi \dd B(s) + \int_{E} \Theta(s,e) \widetilde{N}(\dd e, \dd s),~ s \in [t,T)
\\
\mathcal{P}(T) = \mathbb{M}_1 \\
\det (I - \mathcal{P}(s-) \widetilde{\mathbb{H}}(s)) \neq 0,~ \forall s \in [t,T] \\
\det (\widehat{R}_1(s)) \neq 0,~ \forall s \in [t,T].
\end{cases}	
\end{align}
\begin{remark}\label{Remark_6_1_1_1_1_1_2}
	{\color{black} Due to Assumption \ref{Assumption_4}, we only need the invertibility of $(I - \mathcal{P}(s-) \widetilde{\mathbb{H}})$ to find the expression of $\mathcal{Z}$ and $\mathcal{K}$ in (\ref{eq_3_30}), which is different from the case with Assumption \ref{Assumption_3} in (\ref{eq_3_20}). Note that when there are no jumps, we have $\widehat{\mathscr{F}}_{2} = 0$ ($\mathcal{K} = 0$) in (\ref{eq_3_30}), in which case the ISRDE in (\ref{eq_3_31}) degenerates to the SRDE in \cite[(3.38)]{Jiongmin_Yong_1}. This means that the Four-Step Scheme of \textbf{(LQ-L)} under Assumption \ref{Assumption_4} generalizes \cite[Section 3]{Jiongmin_Yong_1} to the problem of jump-diffusion models.
	}
\end{remark}

Applying (\ref{eq_3_28}), (\ref{eq_3_12}) and  (\ref{eq_3_30}) to $\mathcal{X}$ in (\ref{eq_3_10}) yields
\begin{align}
\label{eq_3_32}
\begin{cases}
\dd \mathcal{X}(s) =   \widetilde{\mathbb{A}}(s)\mathcal{X}(s-) \dd s  + \widetilde{\mathbb{C}}(s)\mathcal{X}(s-)  \dd B(s) \\
~~~~~~~~~~ + \int_{E}  \widetilde{\mathbb{F}}(s,e) \mathcal{X}(s-)   \widetilde{N}(\dd e, \dd s),~ s \in (t,T] \\
\mathcal{X}(t) = \overline{\mathcal{X}},
\end{cases}
\end{align}
where ($s$ is suppressed)
\begin{align*}
\begin{cases}
\widetilde{\mathbb{A}} := \mathbb{A} + \mathbb{B}_2 \mathcal{P}(s-) + \widehat{\mathbb{H}}	 \widehat{\mathscr{F}}_{1} - \mathbb{B}_1 \widehat{\mathcal{R}}_1^{-1} \widehat{\mathcal{H}}_1 \\
\widetilde{\mathbb{C}} := \mathbb{C} + \widehat{\mathbb{H}}^\top \mathcal{P}(s-) + \widetilde{\mathbb{H}} \widehat{\mathscr{F}}_{1} - \mathbb{D}_1 \widehat{\mathcal{R}}_1^{-1} \widehat{\mathcal{H}}_1\\
\widetilde{\mathbb{F}} := \mathbb{F} - \mathbb{G}_1 \widehat{\mathcal{R}}_1^{-1} \widehat{\mathcal{H}}_1.
\end{cases}
\end{align*}

\begin{theorem}\label{Theorem_3}
Suppose that Assumptions \ref{Assumption_1}, \ref{Assumption_2} and \ref{Assumption_4} hold. Let $(\mathcal{P},\Psi,\Theta) \in \mathcal{C}_{\mathbb{F}}^2(t,T;\mathbb{R}^{2n \times 2n}) \times \mathcal{L}_{\mathbb{F}}^2(t,T;\mathbb{R}^{2n \times 2n}) \times \mathcal{L}_{\mathbb{F},p}^2(t,T;\mathbb{R}^{2n \times 2n})$ be the solution of the ISRDE in (\ref{eq_3_31}), and $\mathcal{X}$ the solution of (\ref{eq_3_32}). Define the transformations in (\ref{eq_3_12}) and (\ref{eq_3_30}), and consider the control in (\ref{eq_3_28}). Then (\ref{eq_3_10}) and (\ref{eq_3_11}) hold. In addition, suppose that (\ref{eq_3_3}) holds. Then the state-feedback type control in (\ref{eq_3_28}) is the optimal control for \textbf{(LQ-L)}, and the associated optimal cost is given by
\begin{align*}
J_1(a;\overline{u}_1,\overline{u}_2) = \inf_{u_1 \in \mathcal{U}_1} J_1(a;u_1,\overline{u}_2) = \langle a, \mathcal{P}_{11} (t) a \rangle.
\end{align*}
\end{theorem}

Using (\ref{eq_2_13}) and (\ref{eq_3_28}), we introduce
\begin{align}
\label{eq_3_34}
\begin{cases}
\overline{u}_1(s) = - \widehat{\mathcal{R}}_1(s)^{-1} \widehat{\mathcal{B}}_1(s) \begin{bmatrix}
 x(s-) \\
 \beta(s-)	
 \end{bmatrix} \\
\overline{u}_2(s)   = - \widehat{R}_2(s)^{-1} \widehat{S}_2(s)^\top x(s-) - \widehat{R}_2(s)^{-1} \Bigl (B_2(s)^\top \phi(s-)  \\
~~~ + D_2(s)^\top \theta(s) - \widehat{S}_1(s) \widehat{\mathcal{R}}_1(s)^{-1} \widehat{\mathcal{B}}_1(s) \begin{bmatrix}
 x(s-) \\
 \beta(s-)	
 \end{bmatrix} \Bigr ).
\end{cases}	
\end{align}
\begin{corollary}\label{Corollary_2}
Suppose that the assumptions of Theorems \ref{Theorem_1} and \ref{Theorem_3} hold. Then $(\overline{u}_1,\overline{u}_2) \in \mathcal{U}_1 \times \mathcal{U}_2$ in (\ref{eq_3_34}) constitutes the state-feedback representation of the open-loop Stackelberg equilibrium for the leader and the follower.
\end{corollary}

We state several remarks on the results in Section \ref{Section_3}.
\begin{remark}\label{Remark_7}
\begin{enumerate}[(i)]
	\item The ISRDEs of the leader for two different cases in (\ref{eq_3_21}) and (\ref{eq_3_31}) are not symmetric due to the nonsymmetric coupling nature of $\mathcal{Z}$ and $\mathcal{K}$ in (\ref{eq_3_20}) and (\ref{eq_3_30}). In fact, we can observe that $\mathcal{R}_1$ in (\ref{eq_3_19}) and $\int_{E} \mathbb{F}(s,e)^\top \widehat{\mathscr{F}}_{2}(s,e)  \lambda (\dd e)$ in   (\ref{eq_3_31}) are not symmetric. On the other hand, for the case of SDEs in a Brownian setting without jumps, the corresponding SRDE of the leader in \cite[(3.38)]{Jiongmin_Yong_1} is symmetric.
	\item {\color{black}The well-posedness (solvability) of the ISRDEs of the leader in (\ref{eq_3_21}) and (\ref{eq_3_31}) is a challenging problem. Note that even for the case without jumps, the well-posedness of the (symmetric) SRDE of the leader in \cite[(3.38)]{Jiongmin_Yong_1} has not been solved in the existing literature. We leave the well-posedness of (\ref{eq_3_21}) and (\ref{eq_3_31}) as a future research problem.}
\item It is hard to consider the general situation (without Assumption \ref{Assumption_3} or Assumption \ref{Assumption_4}) to obtain the state-feedback type optimal control of \textbf{(LQ-L)}. Specifically, without Assumption \ref{Assumption_3} or Assumption \ref{Assumption_4}, it is necessary to use (\ref{eq_3_14}) to obtain the expression of $\mathcal{Z}$ and $\mathcal{K}$. Then from (\ref{eq_3_14}), the following holds ($s$ is suppressed):
\begin{align}
\label{eq_4_1}
\begin{cases}
	(I - \mathcal{P} (s-)\widetilde{\mathbb{H}} ) \mathcal{Z}(s)  - \mathcal{P} (s-)\int_{E} \widetilde{\mathbb{K}}(s,e) \mathcal{K}(s,e) \lambda (\dd e) \\
~~~~~	= \mathcal{P}(s-) \bigl [ \mathbb{C}\mathcal{X} + \widehat{\mathbb{H}}^\top\mathcal{P} \mathcal{X}     + \mathbb{D}_1(s) \overline{u}_1 \bigr ]  + \Psi \mathcal{X} \\
\mathcal{K}(s,e) - (\mathcal{P}(s-) + \Theta(s,e)) \bigl ( \int_{E} \overline{\mathbb{K}}(s,e,e^\prime) \mathcal{K}(s,e^\prime) \lambda (\dd e^\prime) + \widetilde{\mathbb{K}}^\top (s,e) \mathcal{Z}(s) \bigr ) \\
~~~~~ = (\mathcal{P}(s-) + \Theta(s,e))  \bigl [ \mathbb{F} \mathcal{X} + \widehat{\mathbb{K}}^\top \mathcal{P}(s-) \mathcal{X}  + \mathbb{G}_1(s,e) \overline{u}_1 \bigr ] + \Theta(s,e) \mathcal{X}.
\end{cases}	
\end{align}
Due to the integral terms $\int_{E} \widetilde{\mathbb{K}}(s,e) \mathcal{K}(s,e) \lambda (\dd e)$ and $\int_{E} \overline{\mathbb{K}}(s,e,e^\prime) \mathcal{K}(s,e^\prime) \lambda (\dd e^\prime)$, and the cross-coupling structure of $\mathcal{Z}$ and $\mathcal{K}$ in (\ref{eq_4_1}), there is a technical challenge to find the explicit expression of $\mathcal{Z}$ and $\mathcal{K}$ in (\ref{eq_4_1}). In (\ref{eq_4_1}), Assumption \ref{Assumption_3} implies $\int_{E} \widetilde{\mathbb{K}}(s,e) \mathcal{K}(s,e) \lambda (\dd e) = \lambda \widetilde{\mathbb{K}}(s) \mathcal{K}(s)$ and $\int_{E} \overline{\mathbb{K}}(s,e,e^\prime) \mathcal{K}(s,e^\prime) \lambda (\dd e^\prime) = \lambda \overline{\mathbb{K}}(s) \mathcal{K}(s)$. Moreover, Assumption \ref{Assumption_4} leads to $\int_{E} \widetilde{\mathbb{K}}(s,e) \mathcal{K}(s,e) \lambda (\dd e) = 0$ and $\int_{E} \overline{\mathbb{K}}(s,e,e^\prime) \mathcal{K}(s,e^\prime) \lambda (\dd e^\prime) = 0$. Hence, in both cases, we are able to find the explicit expressions of $\mathcal{Z}$ and $\mathcal{K}$, which are given in (\ref{eq_3_20}) and (\ref{eq_3_30}).
\end{enumerate}
\end{remark}

\section{Concluding Remarks} \label{Section_4}

{\color{black}We have considered the linear-quadratic (LQ) stochastic (leader-follower) Stackelberg differential game for jump-diffusion systems with random coefficients. Unlike \cite{Jiongmin_Yong_1}, while characterizing the state-feedback type optimal control of the leader, the technical restriction exists (see Remark \ref{Remark_7}). We have identified two different conditions to resolve this restriction. In both cases, the (two different) state-feedback types of the open-loop type Stackelberg equilibrium are characterized by establishing the stochastic maximum principle and Four-Step Schemes with jump diffusions, which are nontrivial extensions of the problem without jumps in \cite{Jiongmin_Yong_1}. We should mention that the LQ stochastic control problem of the leader is new and nontrivial, since  it is the indefinite LQ problem and its constraint is the coupled FBSDEs with jump diffusion and random coefficients induced by the rational behavior of the follower. We have obtained the stochastic maximum principle for the indefinite LQ problem of the leader via the variational approach.
}

We state some special cases of the results of this paper.
\begin{itemize}
\item When there are no jumps in (\ref{eq_1}), i.e., $F = G_1 = G_2 = 0$, Corollaries \ref{Corollary_1} and \ref{Corollary_2} become equivalent, which are reduced to the Stackelberg equilibrium without jumps in \cite{Jiongmin_Yong_1}. Specifically, for the case without jumps, Theorem \ref{Theorem_1} of \textbf{(LQ-F)} coincides with \cite[Theorem 2.3]{Jiongmin_Yong_1}. Moreover, in \textbf{(LQ-L)}, Theorems \ref{Theorem_2} and \ref{Theorem_3} become identical, which degenerate to \cite[Theorem 3.3]{Jiongmin_Yong_1}. 
\item When all the coefficients in (\ref{eq_1})-(\ref{eq_3}) are deterministic, we have $L=Z=\theta=\psi=0$ in (\ref{eq_2_11}) and (\ref{eq_2_12}), and $\Psi = \Theta = 0$ in (\ref{eq_3_21}) and (\ref{eq_3_31}). In this case, the ISRDEs for the leader and the follower become the deterministic integro-Riccati differential equations.
\end{itemize}

There are several interesting potential future research problems of this paper. {\color{black}One is the solvability of the ISRDEs of the leader in (\ref{eq_3_21}) and (\ref{eq_3_31}) for the existence of the Stackelberg equilibrium in view of Corollaries \ref{Corollary_1} and \ref{Corollary_2}.} Another potential problem is the mean-field type problem, in which case the expected values of $x$, $u_1$ and $u_2$, i.e., $\mathbb{E}[x(s)]$, $\mathbb{E}[u_1(s)]$ and $\mathbb{E}[u_2(s)]$, are included in (\ref{eq_1})-(\ref{eq_3}). This problem can be viewed as a generalization of \cite{Lin_TAC_2019} to jump-diffusion models. Finally, it is possible to study the Markov regime-switching jump-diffusion system, for which an additional Markov jump parameter is included in (\ref{eq_1})-(\ref{eq_3}). In this problem, we need to apply (and generalize) the stochastic maximum principle in \cite{Zhang_SICON_2012}.

\bibliographystyle{siamplain}
\bibliography{researches_1.bib}
\end{document}